\documentclass[english]{article}
\pdfoutput=1
\usepackage[T1]{fontenc}
\usepackage[latin9]{inputenc}
\usepackage{geometry}
\usepackage{color}
\usepackage{amsmath}
\usepackage{amsthm}
\usepackage{amssymb}
\usepackage{graphicx}

\makeatletter

\providecommand{\tabularnewline}{\\}

\numberwithin{equation}{section}
\numberwithin{figure}{section}
\theoremstyle{plain}
\newtheorem{thm}{\protect\theoremname}[section]
\theoremstyle{plain}
\newtheorem{prop}[thm]{\protect\propositionname}
\theoremstyle{plain}
\newtheorem{lem}[thm]{\protect\lemmaname}
\newcommand{\lyxaddress}[1]{
	\par {\raggedright #1
	\vspace{1.4em}
	\noindent\par}
}

\DeclareMathOperator{\conv}{conv}

\DeclareMathOperator{\Dom}{Dom}

\DeclareMathOperator{\inter}{int}

\DeclareMathOperator{\relint}{relint}
\DeclareMathOperator{\Ran}{Ran}

\DeclareMathOperator{\vertices}{vert}
\DeclareMathOperator{\Vol}{Vol}

\date{ }

\makeatother

\usepackage{babel}
\providecommand{\lemmaname}{Lemma}
\providecommand{\propositionname}{Proposition}
\providecommand{\theoremname}{Theorem}

\begin{document}
\title{The maximum number of points in the cross-polytope that form a packing
set of a scaled cross-polytope}
\author{Ji Hoon Chun\footnote{The research of the author was supported by the Deutsche Forschungsgemeinschaft (DFG) Graduiertenkolleg "Facets of Complexity/Facetten der Komplexit{\"a}t" (GRK 2434).} }
\maketitle
\begin{abstract}
The problem of finding the largest number of points in the unit cross-polytope
such that the $l_{1}$-distance between any two distinct points is
at least $2r$ is investigated for $r\in\left(1-\frac{1}{n},1\right]$
in dimensions $\geq2$ and for $r\in\left(\frac{1}{2},1\right]$ in
dimension $3$. For the $n$-dimensional cross-polytope, $2n$ points
can be placed when $r\in\left(1-\frac{1}{n},1\right]$. For the three-dimensional
cross-polytope, $10$ and $12$ points can be placed if and only if
$r\in\left(\frac{3}{5},\frac{2}{3}\right]$ and $r\in\left(\frac{4}{7},\frac{3}{5}\right]$
respectively, and no more than $14$ points can be placed when $r\in\left(\frac{1}{2},\frac{4}{7}\right]$.
Also, constructive arrangements of points that attain the upper bounds
of $2n$, $10$, and $12$ are provided, as well as $13$ points for
dimension $3$ when $r\in\left(\frac{1}{2},\frac{6}{11}\right]$.
\end{abstract}

\section{Introduction}

Let $K$ and $L$ be origin-symmetric convex sets in $\mathbb{R}^{n}$
with nonempty interiors. A set $D\subset\mathbb{R}^{n}$ is a (translative)
packing set for $K$ if, for all distinct $\mathbf{x},\mathbf{y}\in D$,
\[
\left(\mathbf{x}+\inter K\right)\cap\left(\mathbf{y}+\inter K\right)=\emptyset.
\]
 For $r>0$, we consider the problem of finding the maximum number
of points in a packing set $D$ of $rK$ that is contained in $L$.
This quantity will be denoted by 
\[
\gamma\left(L,K,r\right):=\max\left\{ \left|D\right|\mid D\subset L\text{ and }\left|\left|\mathbf{x}-\mathbf{y}\right|\right|_{K}\geq2r\text{ for any }\mathbf{x},\mathbf{y}\in D,\mathbf{x}\neq\mathbf{y}\right\} ,
\]
 where $\left|\left|\mathbf{x}\right|\right|_{K}=\min\left\{ \lambda\mid\lambda\geq0\text{ and }\mathbf{x}\in\lambda K\right\} $
and for a set $S$, its cardinality is denoted by $\left|S\right|$.
If $K=L$ then we use the notation $\gamma\left(K,r\right)$ as a
shorthand for $\gamma\left(L,K,r\right)$. We will only deal with
the situation where both $K$ and $L$ are the unit cross-polytope
$C_{n}^{*}=\left\{ \mathbf{x}\in\mathbb{R}^{n}\left|\,\sum_{i=1}^{n}\left|x_{i}\right|\leq1\right.\right\} $.

A set $D$ in $C_{n}^{*}$ with $k$ points such that the $l_{1}$-distance
between any two distinct points is greater than or equal to $2r$
is equivalent to a packing of $rC_{n}^{*}+D$ such that each ball
is contained inside the set $\left(1+r\right)C_{n}^{*}$. Unless otherwise
specified, we will use ``distance'' to mean the $l_{1}$-distance.
The vertices of $C_{n}^{*}$ are the $2n$ unit vectors $\left\{ \pm\mathbf{e}_{i}\mid i\in\left\{ 1,\ldots,n\right\} \right\} $,
so the distance between any two distinct vertices is $2$, which implies
that $\gamma\left(C_{n}^{*},r\right)=1$ for any $r>1$ and $\gamma\left(C_{n}^{*},r\right)\geq2n$
for $r\leq1$.

The case $r=\frac{1}{2}$ is related to the topic of kissing numbers.
The (translative) kissing number $k\left(K\right)$ of a convex body
$K$ is the maximum number of translates of $K$ such that no two
translates overlap each other and each translate touches but does
not overlap with $K$. In the case of $\frac{1}{2}C_{n}^{*}$, we
have 
\[
k\left(\frac{1}{2}C_{n}^{*}\right)=\max\left\{ \left|D\right|\left|\,D\cup\left\{ \mathbf{0}\right\} \text{ is a packing set for }\frac{1}{2}C_{n}^{*}\text{ and }\left(\mathbf{x}+\frac{1}{2}C_{n}^{*}\right)\cap\frac{1}{2}C_{n}^{*}\neq\emptyset\text{ for any }\mathbf{x}\in D\right.\right\} ,
\]
 and since $k\left(K\right)$ is invariant under the scaling of $K$,
\[
k\left(C_{n}^{*}\right)+1=k\left(\frac{1}{2}C_{n}^{*}\right)+1\leq\gamma\left(C_{n}^{*},\frac{1}{2}\right).
\]
 For the cross-polytope, it is known that $k\left(C_{3}^{*}\right)=18$
\cite{LarmanZong1999,Talata1999}. This result implies that $\gamma\left(C_{3}^{*},\frac{1}{2}\right)\geq19$,
however, due to the requirement that one point is the origin, it does
not a priori provide an upper bound for \emph{any} packing set for
$\frac{1}{2}C_{3}^{*}$. An upper bound for the kissing number of
any convex body $K$ is obtained from a result of Hadwiger \cite{Hadwiger1957},
\[
k\left(K\right)\leq3^{n}-1,
\]
 where this inequality is an equality iff $K$ is a parallelepiped.
The cross-polytope is not a parallelepiped, so $k\left(C_{3}^{*}\right)\leq25$,
which results in an upper bound of $\gamma\left(C_{3}^{*},\frac{1}{2}\right)\leq26$.
In the other direction, a well-known result by Swinnerton-Dyer \cite{Swinnerton-Dyer1953}
says that 
\[
k\left(K\right)\geq n^{2}+n,
\]
 and in the case when $K$ is a cross-polytope the lower bound has
been improved to 
\[
k\left(C_{n}^{*}\right)\geq\left(\frac{9}{8}\right)^{\left(1-o\left(1\right)\right)n}
\]
 by \cite{LarmanZong1999} and to 
\[
k\left(C_{n}^{*}\right)\geq1.13488^{\left(1-o\left(1\right)\right)n}
\]
 by \cite{Talata2000}, which is the best known asymptotic lower bound
for the cross-polytope.

We will work with values of $r$ only in the interval $\left(\frac{1}{2},1\right]$
unless otherwise stated. For $r\in\left(1-\frac{1}{n},1\right]$,
the upper bound for the number of points in the cross-polytope such
that the distance between any two distinct points is at least $2r$
is linear in the dimension of the cross-polytope.
\begin{thm}
Let $n\geq2$, then $\gamma\left(C_{n}^{*},r\right)=2n$ for any $r\in\left(1-\frac{1}{n},1\right]$.
Additionally, $\left(1-\frac{1}{n},1\right]$ is the largest possible
interval such that $\gamma\left(C_{n}^{*},r\right)=2n$ for all $r$
in the interval.
\end{thm}

In particular, $\gamma\left(C_{3}^{*},r\right)=6$ for $r\in\left(\frac{2}{3},1\right]$.
The next theorem is specific to the three-dimensional case.
\begin{thm}
In dimension $3$,
\begin{description}
\item [{(a)}] $\gamma\left(C_{3}^{*},r\right)=10$ for any $r\in\left(\frac{3}{5},\frac{2}{3}\right]$,
\item [{(b)}] $\gamma\left(C_{3}^{*},r\right)=12$ for any $r\in\left(\frac{4}{7},\frac{3}{5}\right]$,
and
\item [{(c)}] $\gamma\left(C_{3}^{*},r\right)\leq14$ for $r\in\left(\frac{1}{2},\frac{4}{7}\right]$.
\end{description}
\end{thm}

For the case $n=3$ and $r\in\left(\frac{1}{2},\frac{4}{7}\right]$,
we could not find exact values of $\gamma\left(C_{3}^{*},r\right)$,
but we do have lower bounds. Since $\gamma\left(C_{3}^{*},r'\right)\geq\gamma\left(C_{3}^{*},r\right)$
for $r'<r$, it follows immediately from Theorem 1.2 (b) is $\gamma\left(C_{3}^{*},r\right)\geq12$
for $r\in\left(\frac{1}{2},\frac{4}{7}\right]$. It is possible to
improve this lower bound for a smaller interval of $r$.
\begin{prop}
In dimension $3$, $\gamma\left(C_{3}^{*},r\right)\geq13$ for $r\in\left(\frac{1}{2},\frac{6}{11}\right]$.
\end{prop}

These lower bounds are obtained by explicit constructions. It follows
from Theorems 1.1 and 1.2, Proposition 1.3, and the above discussion
that 
\begin{align*}
\gamma\left(C_{n}^{*},r\right)=1 & \qquad\text{for }r\in\left(1,\infty\right)\text{,}\\
\gamma\left(C_{n}^{*},r\right)=2n & \qquad\text{for }r\in\left(1-\frac{1}{n},1\right]\text{,}\\
\gamma\left(C_{3}^{*},r\right)=10 & \qquad\text{for }r\in\left(\frac{3}{5},\frac{2}{3}\right]\text{,}\\
\gamma\left(C_{3}^{*},r\right)=12 & \qquad\text{for }r\in\left(\frac{4}{7},\frac{3}{5}\right]\text{,}\\
12\leq\gamma\left(C_{3}^{*},r\right)\leq14 & \qquad\text{for }r\in\left(\frac{6}{11},\frac{4}{7}\right]\text{,}\\
12\leq\gamma\left(C_{3}^{*},r\right)\leq14 & \qquad\text{for }r\in\left(\frac{1}{2},\frac{6}{11}\right]\text{, and}\\
19\leq\gamma\left(C_{3}^{*},r\right)\leq26 & \qquad\text{for }r=\frac{1}{2}\text{.}
\end{align*}
 Below is a chart of the results for dimension $3$ in addition to
the upper and lower bounds for $r=\frac{1}{2}$ mentioned above.
\noindent \begin{center}
\includegraphics[scale=0.33]{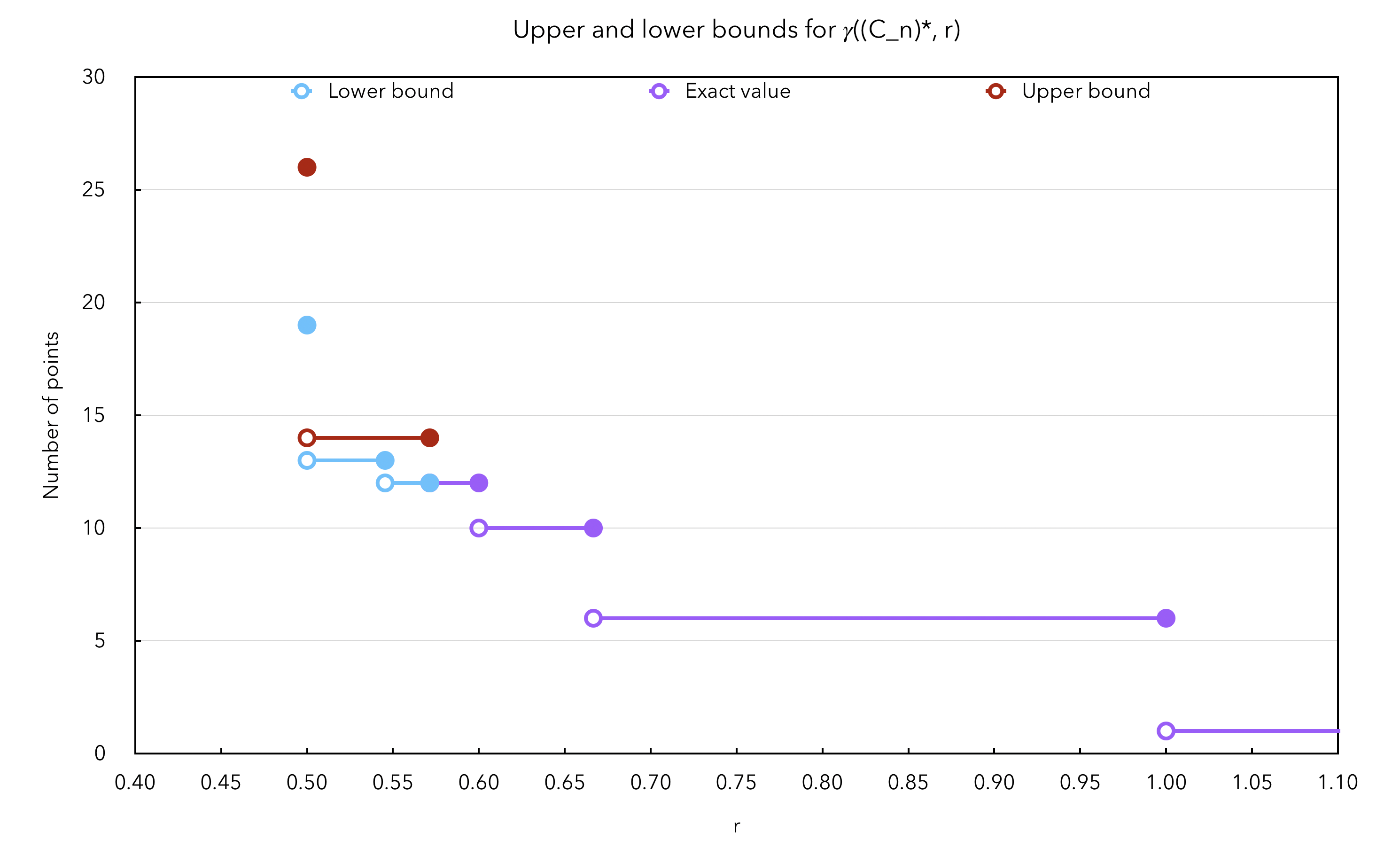}
\par\end{center}

A closely related topic is the problem of packing a set with copies
of another set. This problem has been explored mainly in dimension
$2$. Let $K$ and $L$ be origin-symmetric convex sets with nonempty
interior, then let
\[
M\left(L,K,m\right):=\sup\left\{ r\mid\left|D\right|=m,\ D\subset L\text{, and}\ \left|\left|\mathbf{x}-\mathbf{y}\right|\right|_{K}\geq2r\text{ for any }\mathbf{x},\mathbf{y}\in D,\mathbf{x}\neq\mathbf{y}\right\} .
\]
 From this definition, $M\left(L,K,m\right)$ and $\gamma\left(L,K,r\right)$
are related by the equation 
\[
\gamma\left(L,K,M\left(L,K,m\right)\right)=m.
\]
 The compendium of Goodman, O'Rourke, and T{\'o}th \cite{GoodmanORourkeToth2017}
lists known quantities of $M\left(L,B_{2},m\right)$ for when $L$
is a square, a circle, and an equilateral triangle and various values
of $m$, usually small. In three dimensions, $M\left(L,B_{3},m\right)$
is known for small $m$ and when $L$ is a cube, a cross-polytope,
and a tetrahedron \cite{GoodmanORourkeToth2017}. A related problem
of packings of squares and rectangles in squares is described in \cite{CroftFalconerGuy1991}.
Let $B_{n}$ be the unit Euclidean ball in $n$ dimensions. B{\"o}r{\"o}czky
Jr. and Wintsche have obtained $M\left(C_{n}^{*},B_{n},m\right)$
for $n\geq3$ and $m=\left\{ 3,\ldots,2n+1\right\} $ \cite{BoeroeczkyWintsche2000}.

Let $K$ be a convex set, $B$ be a bounded convex set, $s>0$, and
let $D\left(s,K,B\right)$ be a packing set of $C_{n}^{*}$ such that
$\left|\left\{ \mathbf{x}\in D\left(s,K,B\right)\mid K+\mathbf{x}\subseteq sB\right\} \right|$
is maximal among all packing sets of $C_{n}^{*}$. The\textbf{ }density
of the densest packing of $K$, or the packing density of $K$, is
defined to be 
\[
\delta\left(K\right):=\lim_{s\rightarrow\infty}\frac{\left|\left\{ \mathbf{x}\in D\left(s,K,B\right)\mid K+\mathbf{x}\subseteq sB\right\} \right|\Vol\left(K\right)}{\Vol\left(sB\right)},
\]
 see Definition 4 in Section 20 of \cite{GruberLekkerkerker1987}
(page 225), and it is independent of $B$. Then we can set $B=C_{n}^{*}$
and also suppose that $K=C_{n}^{*}$. For $s>1$, since $C_{n}^{*}+\mathbf{x}\subseteq sC_{n}^{*}$
iff $\mathbf{x}\in\left(s-1\right)C_{n}^{*}$, we have $\left|\left\{ \mathbf{x}\in D\left(s,K,B\right)\mid C_{n}^{*}+\mathbf{x}\subseteq sC_{n}^{*}\right\} \right|=\left|\left\{ \mathbf{x}\in D\left(s,K,B\right)\mid\mathbf{x}\in\left(s-1\right)C_{n}^{*}\right\} \right|$.
Next, scale this set by a factor of $\frac{1}{s-1}$ to get $\left|\left\{ \mathbf{x}\in D\left(s,K,B\right)\mid\mathbf{x}\in\left(s-1\right)C_{n}^{*}\right\} \right|=\left|\left\{ \left.\mathbf{x}\in\frac{1}{s-1}D\left(s,K,B\right)\,\right|\mathbf{x}\in C_{n}^{*}\right\} \right|=\gamma\left(C_{n}^{*},\frac{1}{s-1}\right)$.
Now let $r=\frac{1}{s-1}$. It follows from the definition of the
packing density that 
\[
\delta\left(C_{n}^{*}\right)=\lim_{s\rightarrow\infty}\frac{\gamma\left(C_{n}^{*},\frac{1}{s-1}\right)\Vol C_{n}^{*}}{\Vol sC_{n}^{*}}=\lim_{r\rightarrow0}\gamma\left(C_{n}^{*},r\right)\left(1-\frac{1}{1+r}\right)^{n}.
\]
 Hence the packing density of $C_{n}^{*}$ is related to $\gamma\left(C_{n}^{*},r\right)$
in the sense that $\gamma\left(C_{n}^{*},r\right)\left(1-\frac{1}{1+r}\right)^{n}\sim\delta\left(C_{n}^{*}\right)$
as $r\rightarrow0$.

We now mention some related results involving circle packings in a
circle and sphere packings in a cylinder. For the problem of sphere
packing inside a cylinder of fixed width in three dimensions, Fu et
al. \cite{FuSteinhardtZhaoSocolarCharbonneau2015} predict that as
the radius of the spheres approach zero, densest packings resemble
the face-centered cubic lattice---a densest sphere packing in three
dimensions \cite{ConwaySloane1998}---except for the spheres that
are near the walls of the cylinder. In the case of dimension two the
densest circle packing is generated by the hexagonal lattice \cite{ConwaySloane1998}.
Hopkins, Stillinger, and Torquato \cite{HopkinsStillingerTorquato2010}
provide examples of this phenomenon for dense packings of circles
inside a large circle under the condition that the large circle has
the same center as one of the small circles. Sch{\"u}rmann \cite{Schuermann2002,Schuermann2005}
has shown that under certain conditions the best finite packings of
strictly convex bodies can only be obtained using nonlattice packings.
Other dense arrangements of $k$ circles within a large circle include
modified wedge hexagonal packings and curved hexagonal packings \cite{HopkinsStillingerTorquato2010},
which are the best known packings for some values of $k$ \cite{LubachevskyGraham1997}.

Basic facts about convexity and the cross-polytope can be found in
books such as the ones from Gruber \cite{Gruber2007}, Ziegler \cite{Ziegler2007},
and Coxeter \cite{Coxeter1948}, and about packings are in Conway
and Sloane \cite{ConwaySloane1998}, Gruber \cite{Gruber2007}, and
Zong \cite{Zong1999}. Additional details on the kissing number are
also in Zong \cite{Zong1999}.

Section 2 of this paper provides the notation and preliminaries that
will be used for the rest of the text. Section 3 contains the proof
of the $n$-dimensional case, Theorem 1.1. Section 4 proves the equalities
and upper bound present in the three parts of the $3$-dimensional
case, Theorem 1.2, introducing additional notation as needed. Theorem
1.3 is proved in Section 5, and finally Section 6 presents a gallery
of diagrams related to these lower bounds.

\section{General notation and preliminaries}

Here we introduce notation that will be used over the course of this
paper. For a given $r>0$, let $P_{n}\left(r\right)\subset C_{n}^{*}$
be a packing set of $rC_{n}^{*}$. For any polytope $K$, let $\vertices\left(K\right)$
be the set of its vertices. For a fixed $n\in\mathbb{N}$, define
sets $V_{n}$ and $S_{n}\left(r\right)$ as follows: 
\[
V_{n}:=\vertices\left(C_{n}^{*}\right)=\left\{ \pm\mathbf{e}_{i}\mid i\in\left\{ 1,\ldots,n\right\} \right\} 
\]
 and 
\[
S_{n}\left(r\right):=\left(V_{n}+2r\inter\left(C_{n}^{*}\right)\right)\cap C_{n}^{*}.
\]
 Therefore $S_{n}\left(r\right)$ is the set of all points in $C_{n}^{*}$
that are of distance $<2r$ from some vertex of $C_{n}^{*}$.

For $\mathbf{p}\in\mathbb{R}^{n}$ and $r>0$, we use the notation
\[
C\left(\mathbf{p},r\right):=\left\{ \mathbf{x}\in\mathbb{R}^{n}\mid\left|\left|\mathbf{x}-\mathbf{p}\right|\right|_{1}<r\right\} 
\]
 to denote the interior of the cross-polytope centered at $\mathbf{p}$
and scaled by the factor $r$.

The following lemma is necessary for the general $n$-dimensional
case.
\begin{lem}
Let $r\in\left(0,1\right]$. For each $j\in\left\{ 1,\ldots,n\right\} $,
\[
C_{n}^{*}\cap C\left(\mathbf{e}_{j},2r\right)\subseteq\overline{C\left(\left(1-r\right)\mathbf{e}_{j},r\right)},
\]
 where $\overline{X}$ is the closure of $X$, and similarly for $-\mathbf{e}_{j}$
instead of $\mathbf{e}_{j}$.
\end{lem}

\begin{proof}
Without loss of generality we take the $\mathbf{e}_{j}$ case. Let
$\mathbf{y}=\sum_{i=1}^{n}y_{i}\mathbf{e}_{i}\in C_{n}^{*}\cap C\left(\mathbf{e}_{j},2r\right)$,
then $\left|\left|\mathbf{y}\right|\right|_{1}\leq1$ and $\left|\left|\mathbf{y}-\mathbf{e}_{j}\right|\right|_{1}\leq2r$.
Then the distance from $\mathbf{y}$ to $\left(1-r\right)\mathbf{e}_{j}$
is 
\begin{eqnarray*}
\left|\left|\mathbf{y}-\left(1-r\right)\mathbf{e}_{j}\right|\right|_{1} & = & \left|\left|\sum_{i=1}^{n}y_{i}\mathbf{e}_{i}-\left(1-r\right)\mathbf{e}_{j}\right|\right|_{1}\\
 & = & \left|\left|\sum_{i\neq j}y_{i}\mathbf{e}_{i}+y_{j}\mathbf{e}_{j}-\left(1-r\right)\mathbf{e}_{j}\right|\right|_{1}\\
 & = & \left|\left|\sum_{i\neq j}y_{i}\mathbf{e}_{i}+\left(y_{j}+r-1\right)\mathbf{e}_{j}\right|\right|_{1}\\
 & = & \sum_{i\neq j}\left|y_{i}\right|+\left|y_{j}+r-1\right|.
\end{eqnarray*}
 If $y_{j}+r-1\geq0$ then since $\sum_{i=1}^{n}\left|y_{i}\right|\leq1$,

\begin{eqnarray*}
\sum_{i\neq j}\left|y_{i}\right|+\left|y_{j}+r-1\right| & \leq & \left(1-y_{j}\right)+y_{j}+r-1\\
 & = & r.
\end{eqnarray*}
 Similarly, if $y_{j}+r-1<0$ then since $\left|\left|\mathbf{y}-\mathbf{e}_{j}\right|\right|_{1}\leq2r$,
\begin{eqnarray*}
\sum_{i\neq j}\left|y_{i}\right|+\left|y_{j}+r-1\right| & = & \sum_{i\neq j}\left|y_{i}\right|-y_{j}-r+1\\
 & = & \sum_{i\neq j}\left|y_{i}\right|+\left|y_{j}-1\right|-r\\
 & \leq & 2r-r\\
 & = & r.
\end{eqnarray*}
 So $\left|\left|\mathbf{y}-\left(1-r\right)\mathbf{e}_{j}\right|\right|_{1}\leq r$,
or in other words, $\mathbf{y}\in\left\{ \mathbf{x}\in\mathbb{R}^{n}\mid\left|\left|\mathbf{x}-\left(1-r\right)\mathbf{e}_{j}\right|\right|_{1}\leq r\right\} =\overline{C\left(\left(1-r\right)\mathbf{e}_{j},r\right)}$.
\end{proof}

\section{Proof of Theorem 1.1 (the $\boldsymbol{n}$-dimensional case)}

In this section we assume that $n\geq2$. We will show that for any
$r\in\left(1-\frac{1}{n},1\right]$ and any packing set $P_{n}\left(r\right)$,
the number of points in $P_{n}\left(r\right)\cap S_{n}\left(r\right)$
is bounded above by the number of vertices of $C_{n}^{*}$. Then $\left|P_{n}\left(r\right)\right|\leq2n$
and this inequality is true for all $P_{n}\left(r\right)$, so $\gamma\left(C_{n}^{*},r\right)\leq2n$.
As mentioned in the introduction, the set of vertices $V_{n}\subset C_{n}^{*}$
is a packing set of $rC_{n}^{*}$, which means that $2n$ is also
a lower bound, and so $\gamma\left(C_{n}^{*},r\right)=2n$.
\begin{lem}
Let $r\in\left(1-\frac{1}{n},1\right]$, then $C_{n}^{*}=S_{n}\left(r\right)$.
\end{lem}

\begin{proof}
By definition, $S_{n}\left(r\right)\subseteq C_{n}^{*}$, and so it
remains to show the reverse inclusion. Let $\mathbf{x}\in C_{n}^{*}$
and without loss of generality it can be assumed that $\mathbf{x}$
is in the convex hull of $\mathbf{0},\mathbf{e}_{1},\ldots,\mathbf{e}_{n}$.
Then $\mathbf{x}=\sum_{i=1}^{n}x_{i}\mathbf{e}_{i}$ with $0\leq x_{i}\leq1$
and $\sum_{i=1}^{n}x_{i}\leq1$. Then there exists some $j\in\left\{ 1,\ldots,n\right\} $
such that $x_{j}\geq\frac{1}{n}\sum_{i=1}^{n}x_{i}$, so 
\begin{eqnarray*}
\left|\left|\mathbf{x}-\mathbf{e}_{j}\right|\right|_{1} & = & \left|\left|\sum_{i=1}^{n}x_{i}\mathbf{e}_{i}-\mathbf{e}_{j}\right|\right|_{1}\\
 & = & \left|\left|\sum_{i\neq j}x_{i}\mathbf{e}_{i}+x_{j}\mathbf{e}_{j}-\mathbf{e}_{j}\right|\right|_{1}\\
 & = & \left|\left|\sum_{i\neq j}x_{i}\mathbf{e}_{i}+\left(x_{j}-1\right)\mathbf{e}_{j}\right|\right|_{1}\\
 & = & \sum_{i\neq j}x_{i}+\left(1-x_{j}\right)\\
 & = & \sum_{i=1}^{n}x_{i}-x_{j}+\left(1-x_{j}\right)\\
 & = & \sum_{i=1}^{n}x_{i}-2x_{j}+1\\
 & \leq & \sum_{i=1}^{n}x_{i}-\frac{2}{n}\sum_{i=1}^{n}x_{i}+1\\
 & = & \left(1-\frac{2}{n}\right)\sum_{i=1}^{n}x_{i}+1\\
 & \leq & \left(1-\frac{2}{n}\right)+1\qquad\text{(because }{\textstyle 1-\frac{2}{n}}\geq0\text{)}\\
 & = & 2\left(1-\frac{1}{n}\right)\\
 & < & 2r.
\end{eqnarray*}
 So every point in $C_{n}^{*}$ is within distance $2r$ from some
vertex of $C_{n}^{*}$.
\end{proof}
The next lemma will be crucial for showing that the number of points
in $P_{n}\left(r\right)\cap S_{n}\left(r\right)$ is bounded above
by the number of vertices of $C_{n}^{*}$. It is a uniqueness condition
which shows that if a point $\mathbf{p}\in P_{n}\left(r\right)\cap S_{n}\left(r\right)$
is close to a vertex $\mathbf{v}$ of $C_{n}^{*}$, specifically $\left|\left|\mathbf{p}-\mathbf{v}\right|\right|_{1}<2r$,
then no other point in $P_{n}\left(r\right)$ can be close to $\mathbf{v}$.
\begin{lem}
Let $r\in\left(0,1\right]$. If a vertex $\mathbf{v}$ of $C_{n}^{*}$
has the property that $\mathbf{v}\in C\left(\mathbf{p},2r\right)\cap C\left(\mathbf{q},2r\right)$
for some $\mathbf{p},\mathbf{q}\in P_{n}\left(r\right)\cap S_{n}\left(r\right)$,
then $\mathbf{p}=\mathbf{q}$.
\end{lem}

\begin{proof}
Without loss of generality, let $\mathbf{v}=\mathbf{e}_{j}$ for some
$j\in\left\{ 1,\ldots,n\right\} $, then by hypothesis $\mathbf{e}_{j}\in C\left(\mathbf{p},2r\right)\cap C\left(\mathbf{q},2r\right)$.
It suffices to show that $\left|\left|\mathbf{p}-\mathbf{q}\right|\right|_{1}<2r$
since the distance between two distinct points in $P_{n,r}$ must
be $2r$ or greater. . Then in turn, $\mathbf{p},\mathbf{q}\in C\left(\mathbf{e}_{j},2r\right)$.
Since $C\left(\mathbf{e}_{j},2r\right)$ is open there exists a $r'<r$
($r'$ depends on $\mathbf{p}$ and $\mathbf{q}$) such that $\mathbf{p},\mathbf{q}\in C\left(\mathbf{e}_{j},2r'\right)$.
Then it follows from Lemma 2.1 applied to $C_{n}^{*}\cap C\left(\mathbf{e}_{j},2r'\right)$
that 
\[
C_{n}^{*}\cap C\left(\mathbf{e}_{j},2r'\right)\subseteq\overline{C\left(\left(1-r'\right)\mathbf{e}_{j},r'\right)}\subseteq C\left(\left(1-r'\right)\mathbf{e}_{j},r\right),
\]
 so $\left|\left|\mathbf{p}-\mathbf{q}\right|\right|_{1}<2r$.
\end{proof}
The following lemma will be used both here and in the $3$-dimensional
cases in the next section.
\begin{lem}
Let $r\in\left(0,1\right]$, then 
\[
\left|P_{n}\left(r\right)\cap S_{n}\left(r\right)\right|\leq2n.
\]
\end{lem}

\begin{proof}
Define $V_{n}\left(r\right)$ by 
\[
V_{n}\left(r\right)=\left\{ \mathbf{v}\in V_{n}\mid\text{there exists a }\mathbf{p}\in P_{n}\left(r\right)\cap S_{n}\left(r\right)\text{ such that }\left|\left|\mathbf{v}-\mathbf{p}\right|\right|_{1}<2r\right\} 
\]
 (this set may be empty) and a map $f:V_{n}\left(r\right)\rightarrow P_{n}\left(r\right)\cap S_{n}\left(r\right)$
where $f\left(\mathbf{v}\right)$ is the point $\mathbf{p}\in P_{n}\left(r\right)\cap S_{n}\left(r\right)$
such that $\left|\left|\mathbf{v}-\mathbf{p}\right|\right|_{1}<2r$. 

First we need to show that $f$ is well-defined. Let $\mathbf{p},\mathbf{q}\in P_{n}\left(r\right)\cap S_{n}\left(r\right)$
be points such that $\mathbf{v}\in C\left(\mathbf{p},2r\right)\cap C\left(\mathbf{q},2r\right)$
for some $\mathbf{v}\in V_{n}$, then $\mathbf{p}=\mathbf{q}$ by
Lemma 3.2, which justifies the use of the words ``the point'' in
the definition of $f$. From the definition of $S_{n}\left(r\right)$,
every point $\mathbf{p}\in P_{n}\left(r\right)\cap S_{n}\left(r\right)$
has the property that there is some $\mathbf{v}\in V_{n}$ such that
$\left|\left|\mathbf{p}-\mathbf{v}\right|\right|_{1}<2r$, so $f$
is surjective. Both the domain and range of $f$ are finite sets,
so the cardinality of the range can be bounded above by 
\[
\left|P_{n}\left(r\right)\cap S_{n}\left(r\right)\right|=\left|\Ran\left(f\right)\right|\leq\left|\Dom\left(f\right)\right|\leq\left|V_{n}\left(r\right)\right|=2n,
\]
 completing the proof.
\end{proof}
Now we prove Theorem 1.1. With the preparation above, the proof is
mostly a matter of putting together earlier lemmas.
\begin{proof}[\emph{Proof of Theorem 1.1}]
 Assume that $P_{n}\left(r\right)$ is nonempty, otherwise $\left|P_{n}\left(r\right)\right|=0$
and there is nothing to prove. Since $r\in\left(1-\frac{1}{n},1\right]$,
it follows from Lemma 3.1 that $C_{n}^{*}=S_{n}$, so $\left|P_{n}\left(r\right)\cap S_{n}\left(r\right)\right|$
is nonempty. Then Lemma 3.3 shows that $\left|P_{n}\left(r\right)\right|=\left|P_{n}\left(r\right)\cap S_{n}\left(r\right)\right|\leq2n$.
This inequality holds for any $P_{n}\left(r\right)$, so 
\[
\gamma\left(C_{n}^{*},r\right)\leq2n\qquad\text{for }n\geq2\text{ and }r\in\left(1-\frac{1}{n},1\right]\text{.}
\]
 The upper bound of $2n$ is achieved by $V_{n}=\left\{ \pm\mathbf{e}_{i}\mid i\in\left\{ 1,\ldots,n\right\} \right\} $
as a packing set of $rC_{n}^{*}$, so 
\[
\gamma\left(C_{n}^{*},r\right)=2n\qquad\text{for }n\geq2\text{ and }r\in\left(1-\frac{1}{n},1\right]\text{.}
\]
 The interval $r\in\left(1-\frac{1}{n},1\right]$ cannot be extended
in either direction, because $\gamma\left(C_{n}^{*},r\right)=1$ for
$r>1$ and in Proposition 5.1 we construct a packing set of $rC_{n}^{*}$,
for $r\leq1-\frac{1}{n}$, with $2n+2$ points in $C_{n}^{*}$. For
such $r$, $S_{n}\left(r\right)\subsetneq C_{n}^{*}$ and specifically
the centroid of each facet is not in $S_{n}\left(r\right)$ (cf. Subsection
4.1), so the set consisting of the $2n$ vertices of $C_{n}^{*}$
and the two centroids on opposing facets of $C_{n}^{*}$ is a packing
set of $rC_{n}^{*}$. Therefore, $r\in\left(1-\frac{1}{n},1\right]$
is the largest possible interval such that $\gamma\left(C_{n}^{*},r\right)=2n$
is true.
\end{proof}

\section{Proof of Theorem 1.2 (the $\boldsymbol{3}$-dimensional case)}

When $r\leq\frac{2}{3}$, the set $S_{3}\left(r\right)$ no longer
covers all of $C_{3}^{*}$, so unlike the $n$-dimensional case above,
the proofs for the three-dimensional cases require consideration of
the remainder $C_{3}^{*}\backslash S_{3}\left(r\right)$.

\subsection{Notation and preliminaries for dimension $\boldsymbol{3}$}

Here we collect some lemmas and notation for the three-dimensional
cases. Let $r\in\left[\frac{1}{2},\frac{2}{3}\right]$. Recall that
\[
V_{3}=\vertices\left(C_{3}^{*}\right)=\left\{ \pm\mathbf{e}_{1},\pm\mathbf{e}_{2},\pm\mathbf{e}_{3}\right\} 
\]
 and 
\begin{eqnarray*}
S_{3}\left(r\right) & = & \left(V_{3}+2r\inter\left(C_{n}^{*}\right)\right)\cap C_{n}^{*}\\
 & = & \bigcup_{i=1}^{3}\left(C\left(\mathbf{e}_{i},2r\right)\cup C\left(-\mathbf{e}_{i},2r\right)\right).
\end{eqnarray*}

For any $\sigma_{1},\sigma_{2},\sigma_{3}\in\left\{ -1,1\right\} $,
define the following subsets of $\mathbb{R}^{3}$: 
\[
V\left(r,\left(\sigma_{1},\sigma_{2},\sigma_{3}\right)\right):=\left\{ \begin{pmatrix}\sigma_{1}\left(2r-1\right)\\
\sigma_{2}\left(2r-1\right)\\
\sigma_{3}\left(2r-1\right)
\end{pmatrix},\begin{pmatrix}\sigma_{1}\left(1-r\right)\\
\sigma_{2}\left(1-r\right)\\
\sigma_{3}\left(2r-1\right)
\end{pmatrix},\begin{pmatrix}\sigma_{1}\left(1-r\right)\\
\sigma_{2}\left(2r-1\right)\\
\sigma_{3}\left(1-r\right)
\end{pmatrix},\begin{pmatrix}\sigma_{1}\left(2r-1\right)\\
\sigma_{2}\left(1-r\right)\\
\sigma_{3}\left(1-r\right)
\end{pmatrix}\right\} .
\]

\noindent \begin{center}
\includegraphics[scale=0.3]{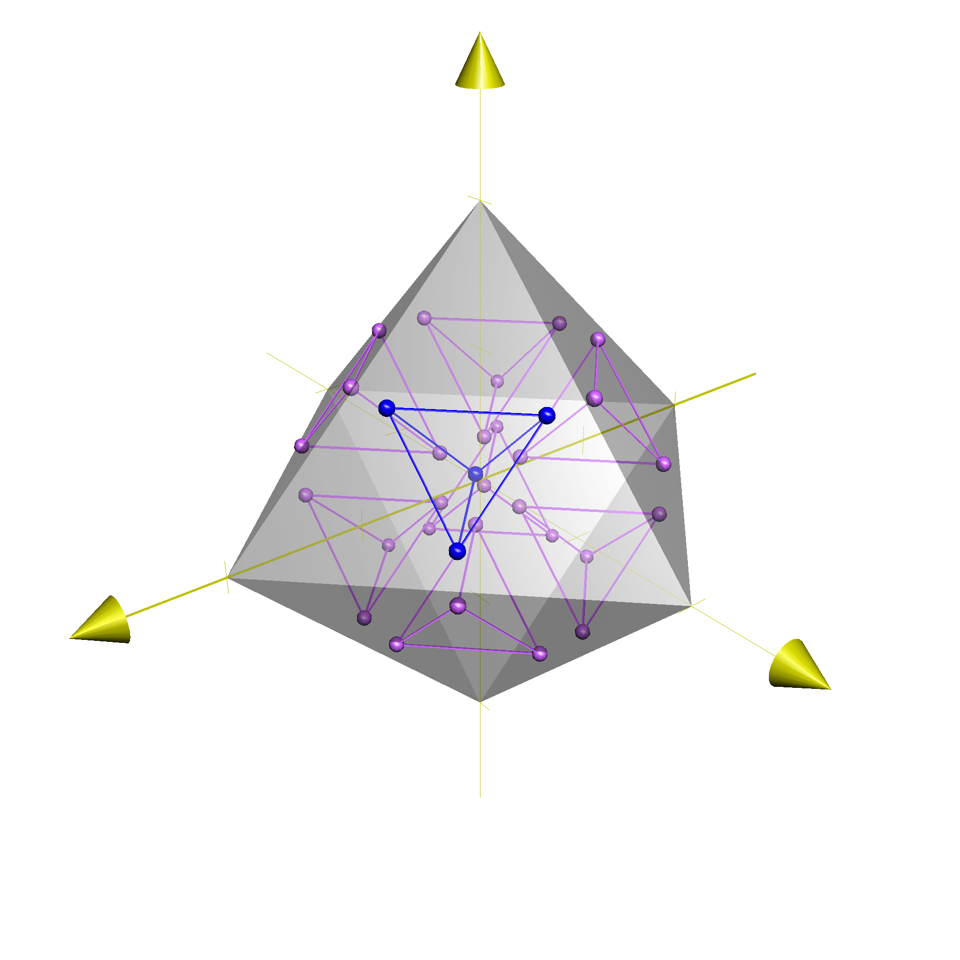}
\par\end{center}

\noindent \begin{center}
\emph{Figure 4.1}. The grey cross-polytope is the set $C_{3}^{*}$,
the blue spheres are the points of $V\left(\frac{11}{20},\left(1,1,1\right)\right)$,
and the purple spheres are the points of $V\left(\frac{11}{20},\left(\sigma_{1},\sigma_{2},\sigma_{3}\right)\right)$
for $\sigma_{1},\sigma_{2},\sigma_{3}\in\left\{ -1,1\right\} $ and
not all equal to $1$.
\par\end{center}

The endpoints of the range $r\in\left[\frac{1}{2},\frac{2}{3}\right]$
are $\frac{2}{3}$ and $\frac{1}{2}$. For $r=\frac{2}{3}$ the set
reduces to 
\[
V\left(\frac{2}{3},\left(1,1,1\right)\right)=\left\{ \begin{pmatrix}\frac{1}{3}\\
\frac{1}{3}\\
\frac{1}{3}
\end{pmatrix}\right\} ,
\]
 the centroid of the facet $\conv\left\{ \mathbf{e}_{1},\mathbf{e}_{2},\mathbf{e}_{3}\right\} $,
and when $r=\frac{1}{2}$ the set is 
\[
V\left(\frac{1}{2},\left(1,1,1\right)\right)=\left\{ \begin{pmatrix}0\\
0\\
0
\end{pmatrix},\begin{pmatrix}\frac{1}{2}\\
\frac{1}{2}\\
0
\end{pmatrix},\begin{pmatrix}\frac{1}{2}\\
0\\
\frac{1}{2}
\end{pmatrix},\begin{pmatrix}0\\
\frac{1}{2}\\
\frac{1}{2}
\end{pmatrix}\right\} ,
\]
 which contains the midpoints of the edges of the facet $\conv\left\{ \mathbf{e}_{1},\mathbf{e}_{2},\mathbf{e}_{3}\right\} $.

Subsets defined using midpoints of edges are used to solve the related
problems of finding upper bounds for $k\left(C_{3}^{*}\right)$ and
$M\left(C_{n}^{*},B_{n},m\right)$ for some values of $n$ and $m$.
To find the kissing number of the cross-polytope, Larman and Zong
\cite{LarmanZong1999} divided the boundary of the cross-polytope
into the union of $18$ subsets including sets of the form 
\[
\relint\left(\left(\frac{1}{2}\mathbf{m}+\frac{1}{2}C_{3}^{*}\right)\cap C_{3}^{*}\right),
\]
 where $\mathbf{m}$ is a midpoint of an edge in $V_{3}$, and showed
that each subset could contain the center of at most one cross-polytope,
resulting in $k\left(C_{3}^{*}\right)\leq18$. Another method to prove
that $k\left(C_{3}^{*}\right)\leq18$ was used by Talata \cite{Talata1999},
who showed that any packing set achieving a kissing number of $18$
must consist of six points on the vertices, six points on the midpoints
of the edges of two opposing facets, and the remaining points on the
hexagon passing through the midpoints of the other edges. B{\"o}r{\"o}czky
Jr. and Wintsche \cite{BoeroeczkyWintsche2000} use sets defined by vertices
and midpoints of edges to determine an upper bound for $M\left(C_{n}^{*},B_{n},m\right)$
where $n\geq3$ and $m\in\left\{ 4,\ldots,2n\right\} $.

For a packing set $P_{3}\left(r\right)$ and a set $\conv\left(V\left(r,\left(\sigma_{1},\sigma_{2},\sigma_{3}\right)\right)\right)$,
$\sigma_{1},\sigma_{2},\sigma_{3}\in\left\{ -1,1\right\} $, call
$\conv\left(V\right)$ a blocked set of $P_{3}\left(r\right)$ if
$\conv\left(V\left(r,\left(\sigma_{1},\sigma_{2},\sigma_{3}\right)\right)\right)$
does not contain any points of $P_{3}\left(r\right)$.

First we show that $C_{3}^{*}\backslash S_{3}\left(r\right)$ can
be written in terms of $V\left(r,\left(\sigma_{1},\sigma_{2},\sigma_{3}\right)\right)$.
\begin{lem}
Let $r\in\left(\frac{1}{2},\frac{2}{3}\right]$. For any $\sigma_{1},\sigma_{2},\sigma_{3}\in\left\{ -1,1\right\} $,
define the following subsets of $C_{3}^{*}$: 
\[
R\left(r,\left(\sigma_{1},\sigma_{2},\sigma_{3}\right)\right):=\left(C_{3}^{*}\backslash S_{3}\left(r\right)\right)\cap\left\{ \mathbf{x}\in\mathbb{R}^{3}\mid\sigma_{1}x_{1},\sigma_{2}x_{2},\sigma_{3}x_{3}\geq0\right\} .
\]
 Then $R\left(r,\left(\sigma_{1},\sigma_{2},\sigma_{3}\right)\right)=\conv\left(V\left(r,\left(\sigma_{1},\sigma_{2},\sigma_{3}\right)\right)\right)$
and 
\[
C_{3}^{*}\backslash S_{3}\left(r\right)=\bigcup_{\sigma_{1},\sigma_{2},\sigma_{3}\in\left\{ -1,1\right\} }\conv\left(V\left(r,\left(\sigma_{1},\sigma_{2},\sigma_{3}\right)\right).\right)
\]
\end{lem}

\begin{proof}
Without loss of generality, assume that $\sigma_{1}=\sigma_{2}=\sigma_{3}=1$,
and we will show that $R\left(r,\left(1,1,1\right)\right)=\conv\left(V\left(r,\left(1,1,1\right)\right)\right)$.
The set $R\left(r,\left(1,1,1\right)\right)$ is the subset of the
unit cross-polytope with all nonnegative coordinates and excluding
the sets $C\left(\mathbf{e}_{i},2r\right)$ for $i\in\left\{ 1,2,3\right\} $,
and the set $\conv\left(V\left(r,\left(1,1,1\right)\right)\right)$
is the intersection of the inequalities $-x_{1}-x_{2}+x_{3}\geq-\left(2r-1\right)$,
$x_{1}-x_{2}-x_{3}\leq-\left(2r-1\right)$, $-x_{1}+x_{2}-x_{3}\leq-\left(2r-1\right)$,
and $x_{1}+x_{2}+x_{3}\leq1$, since the four points in $V\left(r,\left(1,1,1\right)\right)$
satisfy each inequality. We will show that $R\left(r,\left(1,1,1\right)\right)$
is also the intersection of these inequalities.
\noindent \begin{center}
\includegraphics[scale=0.3]{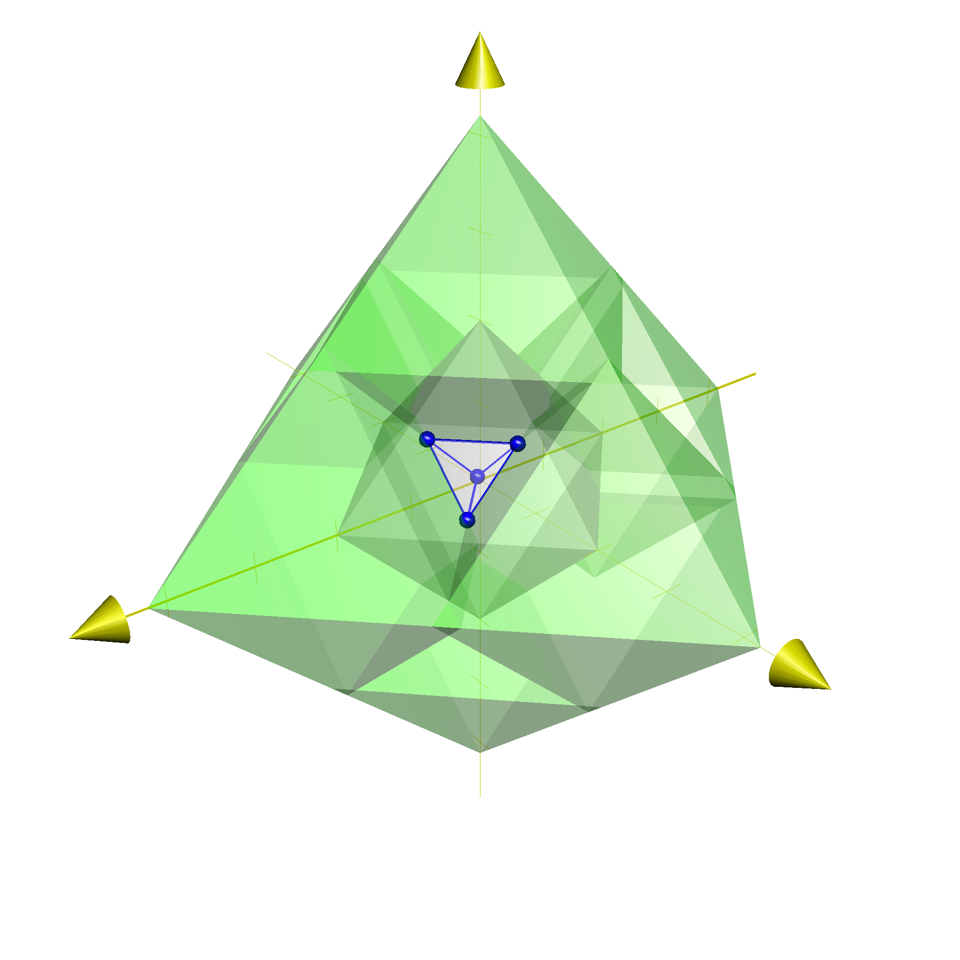}
\par\end{center}

\noindent \begin{center}
\emph{Figure 4.2}. The grey cross-polytope is the set $C_{3}^{*}$,
the green cross-polytopes are the sets $\overline{C\left(\mathbf{x},\frac{11}{20}\right)}$
for $\mathbf{x}\in V_{3}$, and the blue spheres are the points of
$V\left(\frac{11}{20},\left(1,1,1\right)\right)$.
\par\end{center}

Let $\mathbf{x}\in R\left(r,\left(1,1,1\right)\right)$. Then $\mathbf{x}\in C_{3}^{*}$
so $x_{1}+x_{2}+x_{3}\leq1$, and in addition, $\mathbf{x}\notin C\left(\mathbf{e}_{1},2r\right)$
so 
\begin{eqnarray*}
\left|x_{1}-1\right|+\left|x_{2}\right|+\left|x_{3}\right| & \geq & 2r\\
1-x_{1}+x_{2}+x_{3} & \geq & 2r\\
x_{1}-x_{2}-x_{3} & \leq & -\left(2r-1\right).
\end{eqnarray*}
 Similarly, $\mathbf{x}\notin C\left(\mathbf{e}_{2},2r\right)$ and
$\mathbf{x}\notin C\left(\mathbf{e}_{3},2r\right)$ so $-x_{1}+x_{2}-x_{3}\leq-\left(2r-1\right)$
and $x_{1}+x_{2}-x_{3}\geq2r-1$. That proves $R\left(r,\left(1,1,1\right)\right)\subseteq\conv\left(V\left(r,\left(1,1,1\right)\right)\right)$.

For the converse, let $\mathbf{x}\in\conv\left(V\left(r,\left(1,1,1\right)\right)\right)$.
Since $\frac{1}{2}\leq r\leq\frac{2}{3}$, both $2r-1\geq0$ and $1-r\geq0$,
so the four points in $V\left(r,\left(1,1,1\right)\right)$ all have
nonnegative coordinates. Also, $\left|\left|\mathbf{v}\right|\right|_{1}\leq1$
for all $\mathbf{v}\in V\left(r,\left(1,1,1\right)\right)$, and since
$\mathbf{x}$ is in the convex hull of $V\left(r,\left(1,1,1\right)\right)$,
it is also true that $x_{1}+x_{2}+x_{3}\leq1$. Then $x_{1},x_{2},x_{3}\geq0$
and $x_{1}+x_{2}+x_{3}\leq1$ imply that $\mathbf{x}\in C_{n}^{*}$.
Also, $\mathbf{x}$ satisfies $x_{1}-x_{2}-x_{3}\leq-\left(2r-1\right)$,
then 
\begin{eqnarray*}
x_{1}-x_{2}-x_{3} & \leq & -\left(2r-1\right)\\
\left(1-x_{1}\right)+x_{2}+x_{3} & \geq & 2r\\
\left|x_{1}-1\right|+\left|x_{2}\right|+\left|x_{3}\right| & \geq & 2r,
\end{eqnarray*}
 where the last inequality holds because $x_{2},x_{3}\geq0$ and ${\textstyle 0\leq x_{1}\leq\frac{1}{2}}$,
so $\mathbf{x}\notin C\left(\mathbf{e}_{1},2r\right)$. Similarly,
$-x_{1}+x_{2}-x_{3}\leq-\left(2r-1\right)$ and $-x_{1}+x_{2}-x_{3}\leq-\left(2r-1\right)$
so $\mathbf{x}\notin C\left(\mathbf{e}_{2},2r\right)$ and $\mathbf{x}\notin C\left(\mathbf{e}_{3},2r\right)$.
Hence 
\[
\mathbf{x}\in\left(C_{3}^{*}\cap\left\{ x_{1},x_{2},x_{3}\geq0\right\} \right)\left\backslash \left(\bigcup_{i=1}^{3}C\left(\mathbf{e}_{i},2r\right)\right)\right.=R\left(r,\left(1,1,1\right)\right).
\]

To complete the proof of the lemma, note that for any $\mathbf{x}\in\mathbb{R}^{3}$,
let $\sigma_{i}=\frac{x_{i}}{\left|x_{i}\right|}$ if $x_{i}\neq0$
and $\sigma_{i}=1$ if $x_{i}=0$, then $\sigma_{1}x_{1},\sigma_{2}x_{2},\sigma_{3}x_{3}\geq0$,
so $C_{3}^{*}\backslash S_{3}\left(r\right)$ is indeed covered by
all the $R\left(r,\left(\sigma_{1},\sigma_{2},\sigma_{3}\right)\right)$,
$\sigma_{1},\sigma_{2},\sigma_{3}\in\left\{ -1,1\right\} $, resulting
in 
\[
C_{3}^{*}\backslash S_{3}\left(r\right)=\bigcup_{\sigma_{1},\sigma_{2},\sigma_{3}\in\left\{ -1,1\right\} }R\left(r,\left(\sigma_{1},\sigma_{2},\sigma_{3}\right)\right)=\bigcup_{\sigma_{1},\sigma_{2},\sigma_{3}\in\left\{ -1,1\right\} }\conv\left(\left(r,\left(\sigma_{1},\sigma_{2},\sigma_{3}\right)\right)\right).
\]
\end{proof}
From this lemma, the cross-polytope $C_{3}^{*}$ is the union of $S_{3}\left(r\right)$
and the eight regions $\conv\left(V\left(r,\left(\sigma_{1},\sigma_{2},\sigma_{3}\right)\right)\right)$.
By Lemma 3.3, 
\[
\left|P_{3}\left(r\right)\cap S_{3}\left(r\right)\right|\leq6,
\]
 but for $r\in\left(0,\frac{2}{3}\right]$, some points of $P_{3}\left(r\right)$
may be contained in one or more of the sets $\conv\left(V\left(r,\left(\sigma_{1},\sigma_{2},\sigma_{3}\right)\right)\right)$.

For $r\in\left(\frac{1}{2},\frac{2}{3}\right]$, each $V\left(r,\left(\sigma_{1},\sigma_{2},\sigma_{3}\right)\right)$
cannot contain more than one point of $P_{3}\left(r\right)$, which
gives an upper bound of $\left|P_{3}\left(r\right)\right|\leq14$
proved in Subsection 4.4. When $r\in\left(\frac{4}{7},\frac{2}{3}\right]$,
the required minimum distance between points of $P_{3}\left(r\right)$
is large enough so that the presence of a point of $P_{3}\left(r\right)$
in one set $\conv\left(V\left(r\left(\sigma_{1},\sigma_{2},\sigma_{3}\right)\right)\right)$
may imply that one other set $\conv\left(V\left(r\left(\sigma_{1}',\sigma_{2}',\sigma_{3}'\right)\right)\right)$,
$\left(\sigma_{1},\sigma_{2},\sigma_{3}\right)\neq\left(\sigma_{1}',\sigma_{2}',\sigma_{3}'\right)$,
cannot contain any points in $P_{3}\left(r\right)$. Then it is possible
to obtain an upper bound of $12$, and the proof in Subsection 4.3
uses a more complicated argument involving the position of $\mathbf{p}$
in $V\left(r,\left(\sigma_{1},\sigma_{2},\sigma_{3}\right)\right)$.
In Subsection 4.2 we prove that when $r\in\left(\frac{3}{5},\frac{2}{3}\right]$,
a point $\mathbf{p}\in P_{3}\left(r\right)\cap V\left(r,\left(\sigma_{1},\sigma_{2},\sigma_{3}\right)\right)$
implies that three other sets of the form $\conv\left(V\left(r\left(\sigma_{1},\sigma_{2},\sigma_{3}\right)\right)\right)$
cannot contain any points in $P_{3}\left(r\right)$.

\subsection{Proof of Theorem 1.2 (a) (the $\boldsymbol{r\in\left(\frac{3}{5},\frac{2}{3}\right]}$
case)}
\begin{lem}
Let $r\in\left(\frac{3}{5},\frac{2}{3}\right]$ and $\mathbf{p}\in P_{3}\left(r\right)$.
If $\mathbf{p}\in\conv\left(V\left(r,\left(\sigma_{1},\sigma_{2},\sigma_{3}\right)\right)\right)$
for any $\sigma_{1},\sigma_{2},\sigma_{3}\in\left\{ 1,1\right\} $,
then $\conv\left(V\left(r,\left(-\sigma_{1},\sigma_{2},\sigma_{3}\right)\right)\right)$,
$\conv\left(V\left(r,\left(\sigma_{1},-\sigma_{2},\sigma_{3}\right)\right)\right)$,
and $\conv\left(V\left(r,\left(\sigma_{1},\sigma_{2},-\sigma_{3}\right)\right)\right)$
are blocked sets of $P_{3}\left(r\right)$.
\end{lem}

\begin{proof}
Without loss of generality, assume that $\sigma_{1}=\sigma_{2}=\sigma_{3}=1$.
To show that $\conv\left(V\left(r,\left(-\sigma_{1},\sigma_{2},\sigma_{3}\right)\right)\right)$
is a blocked set of $P_{3}\left(r\right)$, it suffices to show that
$\left|\left|\mathbf{p}-\mathbf{y}'\right|\right|_{1}<2r$ for all
\[
\mathbf{y}'\in V\left(r,\left(-1,1,1\right)\right)=\left\{ \begin{pmatrix}-\left(2r-1\right)\\
2r-1\\
2r-1
\end{pmatrix},\begin{pmatrix}-\left(1-r\right)\\
1-r\\
2r-1
\end{pmatrix},\begin{pmatrix}-\left(1-r\right)\\
2r-1\\
1-r
\end{pmatrix},\begin{pmatrix}-\left(2r-1\right)\\
1-r\\
1-r
\end{pmatrix}\right\} ,
\]
 then by the convexity of $V\left(r,\left(-1,1,1\right)\right)$,
the statement $\left|\left|\mathbf{p}-\mathbf{y}\right|\right|_{1}<2r$
holds true for any $\mathbf{y}\in\conv\left(V\left(r,\left(-1,1,1\right)\right)\right)$.
The calculations are as follows: 
\begin{eqnarray*}
\left|\left|\mathbf{p}-\begin{pmatrix}-\left(2r-1\right)\\
2r-1\\
2r-1
\end{pmatrix}\right|\right|_{1} & = & \left|p_{1}+\left(2r-1\right)\right|+\left|p_{2}-\left(2r-1\right)\right|+\left|p_{3}-\left(2r-1\right)\right|\\
 & = & \left(p_{1}+\left(2r-1\right)\right)+\left(\left(2r-1\right)-p_{2}\right)+\left(\left(2r-1\right)-p_{3}\right)\\
 & = & p_{1}-p_{2}-p_{3}+6r-3\\
 & \leq & 1-2p_{2}-2p_{3}+6r-3\\
 & \leq & 1-2\left(2r-1\right)-2\left(2r-1\right)+6r-3\\
 & = & 2-2r\\
 & < & 2r,
\end{eqnarray*}
 
\begin{eqnarray*}
\left|\left|\mathbf{p}-\begin{pmatrix}-\left(1-r\right)\\
1-r\\
2r-1
\end{pmatrix}\right|\right|_{1} & = & \left|p_{1}+\left(1-r\right)\right|+\left|p_{2}-\left(1-r\right)\right|+\left|p_{3}-\left(2r-1\right)\right|\\
 & = & \left(p_{1}+\left(1-r\right)\right)+\left(\left(1-r\right)-p_{2}\right)+\left(\left(2r-1\right)-p_{3}\right)\\
 & = & p_{1}-p_{2}-p_{3}+1\\
 & \leq & 1-2p_{2}-2p_{3}+1\\
 & \leq & 1-2\left(2r-1\right)-2\left(2r-1\right)+1\\
 & = & 6-8r\\
 & < & 2r,
\end{eqnarray*}
 and similarly 
\[
\left|\left|\mathbf{p}-\begin{pmatrix}-\left(1-r\right)\\
2r-1\\
1-r
\end{pmatrix}\right|\right|_{1}<2r
\]
 and 
\begin{eqnarray*}
\left|\left|\mathbf{p}-\begin{pmatrix}-\left(2r-1\right)\\
1-r\\
1-r
\end{pmatrix}\right|\right|_{1} & = & \left|p_{1}+\left(2r-1\right)\right|+\left|p_{2}-\left(1-r\right)\right|+\left|p_{3}-\left(1-r\right)\right|\\
 & = & \left(p_{1}+\left(2r-1\right)\right)+\left(\left(1-r\right)-p_{2}\right)+\left(\left(1-r\right)-p_{3}\right)\\
 & = & p_{1}-p_{2}-p_{3}+1\\
 & < & 2r.
\end{eqnarray*}

By the symmetry of $V\left(r,\left(-1,1,1\right)\right)$, $V\left(r,\left(1,-1,1\right)\right)$,
and $V\left(r,\left(1,1,-1\right)\right)$, it follows that $\left|\left|\mathbf{p}-\mathbf{y}\right|\right|_{1}<2r$
for any $\mathbf{y}\in V\left(r,\left(-1,1,1\right)\right)\cup V\left(r,\left(1,-1,1\right)\right)\cup V\left(r,\left(1,1,-1\right)\right)$,
and so these three sets are blocked sets of $P_{3}\left(r\right)$.
\end{proof}
If $P_{3}\left(r\right)\cap\left(C_{3}^{*}\backslash S_{3}\left(r\right)\right)=\emptyset$
then trivially every set of the form $\conv\left(V\left(r,\left(\sigma_{1},\sigma_{2},\sigma_{3}\right)\right)\right)$,
$\sigma_{1},\sigma_{2},\sigma_{3}\in\left\{ -1,1\right\} $, is a
blocked set of $P_{3}\left(r\right)$. Otherwise, the above lemma
implies that for any given $P_{3}\left(r\right)$, three of the eight
sets of the form $\conv\left(V\left(r,\left(\sigma_{1},\sigma_{2},\sigma_{3}\right)\right)\right)$
are blocked sets of $P_{3}\left(r\right)$. Therefore, 
\[
\left|P_{3}\left(r\right)\cap\left(C_{3}^{*}\backslash S_{3}\left(r\right)\right)\right|\leq5.
\]
 However, it is possible to lower the $5$ to a $4$ with the following
argument.
\begin{lem}
Let $r\in\left(\frac{3}{5},\frac{2}{3}\right]$. Then for any $P_{3}\left(r\right)$,
there exist at least four blocked sets of $P_{3}\left(r\right)$.
\end{lem}

\begin{proof}
Let $\mathbf{p}\in P_{3}\left(r\right)\cap\left(C_{3}^{*}\backslash S_{3}\left(r\right)\right)$.
Without loss of generality assume that there is a $\mathbf{p}\in V\left(r,\left(1,1,1\right)\right)$.
Then by Lemma 4.3, $V\left(r,\left(-1,1,1\right)\right)$, $V\left(r,\left(1,-1,1\right)\right)$,
and $V\left(r,\left(1,1,-1\right)\right)$ are blocked sets of $P_{3}\left(r\right)$.
Consider the set $V\left(r,\left(-1,-1,-1\right)\right)$. If it is
a blocked set, then there is nothing more to prove. If it is not,
then again by Lemma 4.3, $V\left(r,\left(1,-1,-1\right)\right)$,
$V\left(r,\left(-1,1,-1\right)\right)$, and $V\left(r,\left(-1,-1,1\right)\right)$
are blocked sets of $P_{3}\left(r\right)$, resulting in a total of
six blocked sets.
\end{proof}
With this lemma we can prove Theorem 1.2 (a).
\begin{proof}[\emph{Proof of Theorem 1.2 (a)}]
 Let $r\in\left(\frac{3}{5},\frac{2}{3}\right]$. As in the proof
of Theorem 1.2 (c), we split up $P_{3}\left(r\right)$ into $P_{3}\left(r\right)\cap S_{3}\left(r\right)$
and $P_{3}\left(r\right)\cap\left(C_{3}^{*}\backslash S_{3}\left(r\right)\right)$,
then 
\begin{eqnarray*}
\left|P_{3}\left(r\right)\right| & \leq & \left|P_{3}\left(r\right)\cap S_{3}\left(r\right)\right|+\left|P_{3}\left(r\right)\cap\left(C_{3}^{*}\backslash S_{3}\left(r\right)\right)\right|\\
 & \leq & 6+\left|P_{3}\left(r\right)\cap\left(C_{3}^{*}\backslash S_{3}\left(r\right)\right)\right|.
\end{eqnarray*}
 By Lemma 4.4, 
\[
\left|P_{3}\left(r\right)\cap\left(C_{3}^{*}\backslash S_{3}\left(r\right)\right)\right|\leq4,
\]
 which, when combined with the previous inequality, gives 
\begin{eqnarray*}
\left|P_{3}\left(r\right)\right| & \leq & 6+4\\
 & = & 10.
\end{eqnarray*}
 This inequality holds for any $P_{3}\left(r\right)$, so 
\[
\gamma\left(C_{3}^{*},r\right)\leq10\qquad\text{for }r\in\left(\frac{3}{5},\frac{2}{3}\right]\text{.}
\]
 From Proposition 5.2 below, there is a $10$-point packing set for
$rC_{3}^{*}$ contained in $C_{3}^{*}$. So this upper bound is the
best possible, giving 
\[
\gamma\left(C_{3}^{*},r\right)=10\qquad\text{for }r\in\left(\frac{3}{5},\frac{2}{3}\right]\text{.}
\]
\end{proof}

\subsection{Proof of Theorem 1.2 (b) (the $\boldsymbol{r\in\left(\frac{4}{7},\frac{3}{5}\right]}$
case)}

The following additional notation will be used in this section. For
each $r>0$ and $\sigma_{1},\sigma_{2},\sigma_{3}\in\left\{ -1,1\right\} $,
define the following sets $V\left(r,\left(\sigma_{1},\sigma_{2},\sigma_{3}\right)\right)$:
\begin{eqnarray*}
V\left(r,\left(\sigma_{1},\sigma_{2},\sigma_{3}\right),1\right) & := & \left\{ \begin{pmatrix}\sigma_{1}\left(2r-1\right)\\
\sigma_{2}\left(2r-1\right)\\
\sigma_{3}\left(2r-1\right)
\end{pmatrix},\begin{pmatrix}\sigma_{1}\left(2r-1\right)\\
\sigma_{2}\left(1-r\right)\\
\sigma_{3}\left(1-r\right)
\end{pmatrix},\begin{pmatrix}\sigma_{1}\frac{1}{2}r\\
\sigma_{2}\left(1-r\right)\\
\sigma_{3}\frac{1}{2}r
\end{pmatrix},\begin{pmatrix}\sigma_{1}\frac{1}{2}r\\
\sigma_{2}\frac{1}{2}r\\
\sigma_{3}\left(1-r\right)
\end{pmatrix},\frac{1}{3}\begin{pmatrix}\sigma_{1}\\
\sigma_{2}\\
\sigma_{3}
\end{pmatrix}\right\} ,\\
V\left(r,\left(\sigma_{1},\sigma_{2},\sigma_{3}\right),2\right) & := & \left\{ \begin{pmatrix}\sigma_{1}\left(2r-1\right)\\
\sigma_{2}\left(2r-1\right)\\
\sigma_{3}\left(2r-1\right)
\end{pmatrix},\begin{pmatrix}\sigma_{1}\left(1-r\right)\\
\sigma_{2}\left(2r-1\right)\\
\sigma_{3}\left(1-r\right)
\end{pmatrix},\begin{pmatrix}\sigma_{1}\left(1-r\right)\\
\sigma_{2}\frac{1}{2}r\\
\sigma_{3}\frac{1}{2}r
\end{pmatrix},\begin{pmatrix}\sigma_{1}\frac{1}{2}r\\
\sigma_{2}\frac{1}{2}r\\
\sigma_{3}\left(1-r\right)
\end{pmatrix},\frac{1}{3}\begin{pmatrix}\sigma_{1}\\
\sigma_{2}\\
\sigma_{3}
\end{pmatrix}\right\} \text{, and}\\
V\left(r,\left(\sigma_{1},\sigma_{2},\sigma_{3}\right),3\right) & := & \left\{ \begin{pmatrix}\sigma_{1}\left(2r-1\right)\\
\sigma_{2}\left(2r-1\right)\\
\sigma_{3}\left(2r-1\right)
\end{pmatrix},\begin{pmatrix}\sigma_{1}\left(1-r\right)\\
\sigma_{2}\left(1-r\right)\\
\sigma_{3}\left(2r-1\right)
\end{pmatrix},\begin{pmatrix}\sigma_{1}\left(1-r\right)\\
\sigma_{2}\frac{1}{2}r\\
\sigma_{3}\frac{1}{2}r
\end{pmatrix},\begin{pmatrix}\sigma_{1}\frac{1}{2}r\\
\sigma_{2}\left(1-r\right)\\
\sigma_{3}\frac{1}{2}r
\end{pmatrix},\frac{1}{3}\begin{pmatrix}\sigma_{1}\\
\sigma_{2}\\
\sigma_{3}
\end{pmatrix}\right\} .
\end{eqnarray*}
 They have the property that 
\[
\bigcup_{i=1}^{3}\conv\left(V\left(r,\left(\sigma_{1},\sigma_{2},\sigma_{3}\right),i\right)\right)=\conv\left(V\left(r,\left(\sigma_{1},\sigma_{2},\sigma_{3}\right)\right)\right),
\]
 and the numbering of these subsets is so that the set $V\left(r,\left(\sigma_{1},\sigma_{2},\sigma_{3}\right),i\right)$
contains the point in the set 
\[
\left\{ \begin{pmatrix}\sigma_{1}\left(1-r\right)\\
\sigma_{2}\left(1-r\right)\\
\sigma_{3}\left(2r-1\right)
\end{pmatrix},\begin{pmatrix}\sigma_{1}\left(1-r\right)\\
\sigma_{2}\left(2r-1\right)\\
\sigma_{3}\left(1-r\right)
\end{pmatrix},\begin{pmatrix}\sigma_{1}\left(2r-1\right)\\
\sigma_{2}\left(1-r\right)\\
\sigma_{3}\left(1-r\right)
\end{pmatrix}\right\} \subsetneq V\left(r,\left(\sigma_{1},\sigma_{2},\sigma_{3}\right)\right)
\]
 that is furthest away from the vertex $\sigma_{i}\mathbf{e}_{i}$.
\begin{lem}
Let $r\in\left(\frac{4}{7},\frac{3}{5}\right]$ and $\mathbf{p}\in P_{3}\left(r\right)$.
If $\mathbf{x}\in\conv\left(V\left(r,\left(\sigma_{1},\sigma_{2},\sigma_{3}\right)\right)\right)$
then there is a blocked set $\conv\left(V\left(r,\left(\sigma_{1}',\sigma_{2}',\sigma_{3}'\right)\right)\right)$
of $P_{3}\left(r\right)$, with $\left(\sigma_{1},\sigma_{2},\sigma_{3}\right)$
and $\left(\sigma_{1}',\sigma_{2}',\sigma_{3}'\right)$ differing
by exactly one coordinate.
\end{lem}

\begin{proof}
Without loss of generality, assume that $\sigma_{1}=\sigma_{2}=\sigma_{3}=1$,
then $\mathbf{p}$ is in one of the subsets $\conv\left(V\left(r,\left(1,1,1\right),i\right)\right)$
for $i\in\left\{ 1,2,3\right\} $. Assume that $\mathbf{p}\in\conv\left(V\left(r,\left(1,1,1\right),1\right)\right)$
and write $\mathbf{p}=\sum_{i=1}^{3}p_{i}\mathbf{e}_{i}$. We will
show that $\left|\left|\mathbf{x}-\mathbf{y}\right|\right|_{1}<2r$
for any $\mathbf{x}\in V\left(r,\left(1,1,1\right),1\right)$ and
$\mathbf{y}\in V\left(r,\left(-1,1,1\right)\right)$, and then by
the convexity of $\conv\left(V\left(r,\left(1,1,1\right),1\right)\right)$
and $\conv\left(V\left(r,\left(-1,1,1\right)\right)\right)$, it follows
that $\left|\left|\mathbf{p}-\mathbf{y}\right|\right|_{1}<2r$ for
any $\mathbf{y}\in\conv\left(V\left(r,\left(-1,1,1\right)\right)\right)$,
which shows that $\conv\left(V\left(r,\left(-1,1,1\right)\right)\right)$
is a blocked set of $P_{3}\left(r\right)$. This approach is similar
to the proof of the previous lemma, but the same approach cannot be
used here as the second calculation in the proof of Lemma 4.3 ends
with $6-8r<2r$, which is not true for $r\leq\frac{3}{5}$. There
are $20$ different combinations of points but not all of them need
to be explicitly checked. To keep track of the cases, we use the following
grid:
\noindent \begin{center}
\begin{tabular}{|c|c||c|c|c|c|}
\hline 
\multicolumn{1}{|c}{} &  & \multicolumn{4}{c|}{Elements of $V\left(r,\left(-1,1,1\right)\right)$}\tabularnewline
\cline{3-6} \cline{4-6} \cline{5-6} \cline{6-6} 
\multicolumn{1}{|c}{} &  & ${\scriptstyle \begin{pmatrix}{\scriptscriptstyle -\left(2r-1\right)}\\
{\scriptstyle 2r-1}\\
{\scriptstyle 2r-1}
\end{pmatrix}}$ & ${\scriptstyle \begin{pmatrix}{\scriptstyle -\left(1-r\right)}\\
{\scriptstyle 1-r}\\
{\scriptstyle 2r-1}
\end{pmatrix}}$ & ${\scriptstyle \begin{pmatrix}{\scriptstyle -\left(1-r\right)}\\
{\scriptstyle 2r-1}\\
{\scriptstyle 1-r}
\end{pmatrix}}$ & ${\scriptstyle \begin{pmatrix}{\scriptstyle -\left(2r-1\right)}\\
{\scriptstyle 1-r}\\
{\scriptstyle 1-r}
\end{pmatrix}}$\tabularnewline
\hline 
\hline 
\multicolumn{1}{|c|}{} & ${\scriptstyle \begin{pmatrix}{\scriptstyle 2r-1}\\
{\scriptstyle 2r-1}\\
{\scriptstyle 2r-1}
\end{pmatrix}}$ & Case 1 & Case 2 & Case 3 & Case 4\tabularnewline
\cline{2-6} \cline{3-6} \cline{4-6} \cline{5-6} \cline{6-6} 
 & ${\scriptstyle \begin{pmatrix}{\scriptstyle 2r-1}\\
{\scriptstyle 1-r}\\
{\scriptstyle 1-r}
\end{pmatrix}}$ & Case 5 & Case 6 & Case 7 & Case 8\tabularnewline
\cline{2-6} \cline{3-6} \cline{4-6} \cline{5-6} \cline{6-6} 
Elements of $V\left(r,\left(1,1,1\right),1\right)$ & ${\scriptstyle \begin{pmatrix}{\scriptstyle \frac{1}{2}r}\\
{\scriptstyle 1-r}\\
{\scriptstyle \frac{1}{2}r}
\end{pmatrix}}$ & Case 9 & Case 10 & Case 11 & Case 12\tabularnewline
\cline{2-6} \cline{3-6} \cline{4-6} \cline{5-6} \cline{6-6} 
 & ${\scriptstyle \begin{pmatrix}{\scriptstyle \frac{1}{2}r}\\
{\scriptstyle \frac{1}{2}r}\\
{\scriptstyle 1-r}
\end{pmatrix}}$ & Case 13 & Case 14 & Case 15 & Case 16\tabularnewline
\cline{2-6} \cline{3-6} \cline{4-6} \cline{5-6} \cline{6-6} 
 & ${\scriptstyle \begin{pmatrix}{\scriptstyle \frac{1}{3}}\\
{\scriptstyle \frac{1}{3}}\\
{\scriptstyle \frac{1}{3}}
\end{pmatrix}}$ & Case 17 & Case 18 & Case 19 & Case 20\tabularnewline
\hline 
\end{tabular}
\par\end{center}

\noindent For each $k\in\left\{ 1,\ldots,20\right\} $, case $k$
corresponds to the calculation of $\left|\left|\mathbf{x}-\mathbf{y}\right|\right|_{1}$,
where $\mathbf{x}$ is the element of $V\left(r,\left(1,1,1\right),1\right)$
in the same row as $k$ and $\mathbf{y}$ is the element of $V\left(r,\left(-1,1,1\right)\right)$
in the same column as $k$. For example, $\left|\left|\left(2r-1,2r-1,2r-1\right)^{\mathsf{T}}-\left(-\left(2r-1\right),2r-1,2r-1\right)^{\mathsf{T}}\right|\right|_{1}$
will be calculated in case 1 below. Cases that are similar to previous
cases will be pointed out as they arise.
\end{proof}
\begin{enumerate}
\item Since $2r-1,1-r<\frac{1}{2}<r$, it immediately follows that 
\[
\left|\left|\begin{pmatrix}2r-1\\
2r-1\\
2r-1
\end{pmatrix}-\begin{pmatrix}-\left(2r-1\right)\\
2r-1\\
2r-1
\end{pmatrix}\right|\right|_{1}<2r.
\]
\item 
\begin{eqnarray*}
\left|\left|\begin{pmatrix}2r-1\\
2r-1\\
2r-1
\end{pmatrix}-\begin{pmatrix}-\left(1-r\right)\\
1-r\\
2r-1
\end{pmatrix}\right|\right|_{1} & = & \left|\left(2r-1\right)+\left(1-r\right)\right|+\left|\left(2r-1\right)-\left(1-r\right)\right|\\
 & = & \left(\left(2r-1\right)+\left(1-r\right)\right)+\left(\left(1-r\right)-\left(2r-1\right)\right)\\
 & = & 2-2r\\
 & < & 2r.
\end{eqnarray*}
\item By symmetry, this case is similar to case 2.
\item 
\begin{eqnarray*}
\left|\left|\begin{pmatrix}2r-1\\
2r-1\\
2r-1
\end{pmatrix}-\begin{pmatrix}-\left(2r-1\right)\\
1-r\\
1-r
\end{pmatrix}\right|\right|_{1} & = & \left|\left(2r-1\right)+\left(2r-1\right)\right|+\left|\left(2r-1\right)-\left(1-r\right)\right|+\left|\left(2r-1\right)-\left(1-r\right)\right|\\
 & = & \left(\left(2r-1\right)+\left(2r-1\right)\right)+\left(\left(1-r\right)-\left(2r-1\right)\right)+\left(\left(1-r\right)-\left(2r-1\right)\right)\\
 & = & 2-2r\\
 & < & 2r.
\end{eqnarray*}
\item 
\begin{eqnarray*}
\left|\left|\begin{pmatrix}2r-1\\
1-r\\
1-r
\end{pmatrix}-\begin{pmatrix}-\left(1-r\right)\\
2r-1\\
2r-1
\end{pmatrix}\right|\right|_{1} & = & \left|\left(2r-1\right)+\left(1-r\right)\right|+\left|\left(2r-1\right)-\left(1-r\right)\right|+\left|\left(2r-1\right)-\left(1-r\right)\right|\\
 & = & \left(\left(2r-1\right)+\left(1-r\right)\right)+\left(\left(1-r\right)-\left(2r-1\right)\right)+\left(\left(1-r\right)-\left(2r-1\right)\right)\\
 & = & 4-5r\\
 & < & 2r.
\end{eqnarray*}
\item 
\begin{eqnarray*}
\left|\left|\begin{pmatrix}2r-1\\
1-r\\
1-r
\end{pmatrix}-\begin{pmatrix}-\left(1-r\right)\\
1-r\\
2r-1
\end{pmatrix}\right|\right|_{1} & = & \left|\left(2r-1\right)+\left(1-r\right)\right|+\left|\left(2r-1\right)-\left(1-r\right)\right|\\
 & = & \left(\left(2r-1\right)+\left(1-r\right)\right)+\left(\left(1-r\right)-\left(2r-1\right)\right)\\
 & = & 2-2r\\
 & < & 2r.
\end{eqnarray*}
\item This case follows from case 6 due to symmetry.
\item This case follows from the same argument used in case 1, that $2r-1,1-r<\frac{1}{2}<r$.
\item 
\begin{eqnarray*}
\left|\left|\begin{pmatrix}\frac{1}{2}r\\
1-r\\
\frac{1}{2}r
\end{pmatrix}-\begin{pmatrix}-\left(2r-1\right)\\
2r-1\\
2r-1
\end{pmatrix}\right|\right|_{1} & = & \left|\frac{1}{2}r+\left(2r-1\right)\right|+\left|\left(1-r\right)-\left(2r-1\right)\right|+\left|\frac{1}{2}r-\left(2r-1\right)\right|\\
 & = & \left(\frac{1}{2}r+\left(2r-1\right)\right)+\left(\left(1-r\right)-\left(2r-1\right)\right)+\left(\frac{1}{2}r-\left(2r-1\right)\right)\\
 & = & 2-2r\\
 & < & 2r.
\end{eqnarray*}
\item 
\begin{eqnarray*}
\left|\left|\begin{pmatrix}\frac{1}{2}r\\
1-r\\
\frac{1}{2}r
\end{pmatrix}-\begin{pmatrix}-\left(1-r\right)\\
1-r\\
2r-1
\end{pmatrix}\right|\right|_{1} & = & \left|\frac{1}{2}r+\left(2r-1\right)\right|+\left|\frac{1}{2}r-\left(2r-1\right)\right|\\
 & = & \left(\frac{1}{2}r+\left(2r-1\right)\right)+\left(\frac{1}{2}r-\left(2r-1\right)\right)\\
 & = & r\\
 & < & 2r.
\end{eqnarray*}
\item 
\begin{eqnarray*}
\left|\left|\begin{pmatrix}\frac{1}{2}r\\
1-r\\
\frac{1}{2}r
\end{pmatrix}-\begin{pmatrix}-\left(2r-1\right)\\
2r-1\\
1-r
\end{pmatrix}\right|\right|_{1} & = & \left|\frac{1}{2}r+\left(2r-1\right)\right|+\left|\left(1-r\right)-\left(2r-1\right)\right|+\left|\frac{1}{2}r-\left(1-r\right)\right|\\
 & = & \left(\frac{1}{2}r+\left(2r-1\right)\right)+\left(\left(1-r\right)-\left(2r-1\right)\right)+\left(\frac{1}{2}r-\left(1-r\right)\right)\\
 & = & r\\
 & < & 2r.
\end{eqnarray*}
\item 
\begin{eqnarray*}
\left|\left|\begin{pmatrix}\frac{1}{2}r\\
1-r\\
\frac{1}{2}r
\end{pmatrix}-\begin{pmatrix}-\left(2r-1\right)\\
1-r\\
1-r
\end{pmatrix}\right|\right|_{1} & = & \left|\frac{1}{2}r+\left(2r-1\right)\right|+\left|\frac{1}{2}r-\left(1-r\right)\right|\\
 & = & \left(\frac{1}{2}r+\left(2r-1\right)\right)+\left(\frac{1}{2}r-\left(1-r\right)\right)\\
 & = & 4r-2\\
 & < & 2r.
\end{eqnarray*}
\item By symmetry, this case is similar to case 9.
\item By symmetry, this case is similar to case 11.
\item By symmetry, this case is similar to case 10.
\item By symmetry, this case is similar to case 12.
\item 
\begin{eqnarray*}
\left|\left|\begin{pmatrix}\frac{1}{3}\\
\frac{1}{3}\\
\frac{1}{3}
\end{pmatrix}-\begin{pmatrix}-\left(2r-1\right)\\
2r-1\\
2r-1
\end{pmatrix}\right|\right|_{1} & = & \left|\frac{1}{3}+\left(2r-1\right)\right|+\left|\frac{1}{3}-\left(2r-1\right)\right|+\left|\frac{1}{3}-\left(2r-1\right)\right|\\
 & = & \left(\frac{1}{3}+\left(2r-1\right)\right)+\left(\frac{1}{3}-\left(2r-1\right)\right)+\left(\frac{1}{3}-\left(2r-1\right)\right)\\
 & = & \frac{4}{3}-2r\\
 & < & 2r.
\end{eqnarray*}
\item 
\begin{eqnarray*}
\left|\left|\begin{pmatrix}\frac{1}{3}\\
\frac{1}{3}\\
\frac{1}{3}
\end{pmatrix}-\begin{pmatrix}-\left(1-r\right)\\
1-r\\
2r-1
\end{pmatrix}\right|\right|_{1} & = & \left|\frac{1}{3}+\left(2r-1\right)\right|+\left|\frac{1}{3}-\left(1-r\right)\right|+\left|\frac{1}{3}-\left(2r-1\right)\right|\\
 & = & \left(\frac{1}{3}+\left(2r-1\right)\right)+\left(\left(1-r\right)-\frac{1}{3}\right)+\left(\frac{1}{3}-\left(2r-1\right)\right)\\
 & = & \frac{4}{3}-r\\
 & < & 2r.
\end{eqnarray*}
\item By symmetry, this case is similar to case 18.
\item 
\begin{eqnarray*}
\left|\left|\begin{pmatrix}\frac{1}{3}\\
\frac{1}{3}\\
\frac{1}{3}
\end{pmatrix}-\begin{pmatrix}-\left(2r-1\right)\\
1-r\\
1-r
\end{pmatrix}\right|\right|_{1} & = & \left|\frac{1}{3}+\left(2r-1\right)\right|+\left|\frac{1}{3}-\left(1-r\right)\right|+\left|\frac{1}{3}-\left(1-r\right)\right|\\
 & = & \left(\frac{1}{3}+\left(2r-1\right)\right)+\left(\left(1-r\right)-\frac{1}{3}\right)+\left(\left(1-r\right)-\frac{1}{3}\right)\\
 & = & \frac{2}{3}\\
 & < & 2r.
\end{eqnarray*}
\end{enumerate}
\begin{proof}
Hence $\conv\left(V\left(r,\left(-1,1,1\right)\right)\right)$ is
a blocked set of $P_{3}\left(r\right)$. By symmetry, if $\mathbf{p}\in\conv\left(V\left(r,\left(1,1,1\right),2\right)\right)$
or $\mathbf{p}\in\conv\left(V\left(r,\left(1,1,1\right),3\right)\right)$
then calculations similar to the above can be performed with $\mathbf{y}\in\conv\left(V\left(r,\left(1,-1,1\right)\right)\right)$
or $\mathbf{y}\in\conv\left(V\left(r,\left(1,1,-1\right)\right)\right)$
respectively.
\end{proof}
Using the above lemma and the same argument as after Lemma 4.3 in
the last subsection, we have 
\[
\left|P_{3}\left(r\right)\cap\left(C_{3}^{*}\backslash S_{3}\left(r\right)\right)\right|\leq7.
\]
 However, just like in the previous subsection it is possible to lower
the $7$ to a $6$ with the following argument.
\begin{lem}
Let $r\in\left(\frac{4}{7},\frac{3}{5}\right]$. Then for any $P_{3}\left(r\right)$,
there exist at least two blocked sets of $P_{3}\left(r\right)$.
\end{lem}

\begin{proof}
Let $\conv\left(V\left(r,\left(\sigma_{1},\sigma_{2},\sigma_{3}\right)\right)\right)$,
$\sigma_{1},\sigma_{2},\sigma_{3}\in\left\{ -1,1\right\} $ be a blocked
set of $P_{3}\left(r\right)$ and consider the set $\conv\left(V\left(r,\left(-\sigma_{1},-\sigma_{2},-\sigma_{3}\right)\right)\right)$.
If $\conv\left(V\left(r,\left(-\sigma_{1},-\sigma_{2},-\sigma_{3}\right)\right)\right)$
is a blocked set of $P_{3}\left(r\right)$ then we are done. Otherwise,
by Lemma 4.3 there must be a blocked set $\conv\left(V\left(r,\left(\sigma_{1}',\sigma_{2}',\sigma_{3}'\right)\right)\right)$
of $P_{3}\left(r\right)$ such that $\left(-\sigma_{1},-\sigma_{2},-\sigma_{3}\right)$
and $\left(\sigma_{1}',\sigma_{2}',\sigma_{3}'\right)$ differ by
exactly one coordinate. Then $\left(\sigma_{1}',\sigma_{2}',\sigma_{3}'\right)\neq\left(\sigma_{1},\sigma_{2},\sigma_{3}\right)$,
which means that $\conv\left(V\left(r,\left(\sigma_{1},\sigma_{2},\sigma_{3}\right)\right)\right)$
and $\conv\left(V\left(r,\left(\sigma_{1}',\sigma_{2}',\sigma_{3}'\right)\right)\right)$
are two distinct blocked sets of $P_{3}\left(r\right)$.
\end{proof}
The proof of Theorem 1.2 (b) is virtually identical to the proof of
Theorem 1.2 (a).
\begin{proof}[\emph{Proof of Theorem 1.2 (b)}]
 Let $r\in\left(\frac{4}{7},\frac{3}{5}\right]$. As in the proof
of Theorem 1.2 (c), we split up $P_{3}\left(r\right)$ into $P_{3}\left(r\right)\cap S_{3}\left(r\right)$
and $P_{3}\left(r\right)\cap\left(S_{3}^{*}\backslash S_{3}\left(r\right)\right)$,
then 
\begin{eqnarray*}
\left|P_{3}\left(r\right)\right| & \leq & \left|P_{3}\left(r\right)\cap S_{3}\left(r\right)\right|+\left|P_{3}\left(r\right)\cap\left(S_{3}^{*}\backslash S_{3}\left(r\right)\right)\right|\\
 & \leq & 6+\left|P_{3}\left(r\right)\cap\left(S_{3}^{*}\backslash S_{3}\left(r\right)\right)\right|.
\end{eqnarray*}
 By Lemma 4.6, 
\[
\left|P_{3}\left(r\right)\cap\left(S_{3}^{*}\backslash S_{3}\left(r\right)\right)\right|\leq6,
\]
 which, when combined with the previous inequality, gives 
\begin{eqnarray*}
\left|P_{3}\left(r\right)\right| & \leq & 6+6\\
 & = & 12.
\end{eqnarray*}
 This inequality holds for any $P_{3}\left(r\right)$, so 
\[
\gamma\left(C_{3}^{*},r\right)\leq12\qquad\text{for }r\in\left(\frac{4}{7},\frac{3}{5}\right]\text{.}
\]
 From Proposition 5.3 below, there is a $12$-point packing set for
$rC_{3}^{*}$ contained in $C_{3}^{*}$. So this upper bound is the
best possible, giving 
\[
\gamma\left(C_{3}^{*},r\right)=12\qquad\text{for }r\in\left(\frac{4}{7},\frac{3}{5}\right]\text{.}
\]
\end{proof}

\subsection{Proof of Theorem 1.2 (c) (the $\boldsymbol{r\in\left(\frac{1}{2},\frac{4}{7}\right]}$
case)}

For $r\in\left(\frac{1}{2},\frac{4}{7}\right]$, we will use an approach
that has similarities to Larman and Zong \cite{LarmanZong1999} and
B{\"o}r{\"o}czky Jr. and Wintsche \cite{BoeroeczkyWintsche2000} in that the maximum
distance between any two points in $\conv\left(V\left(r,\left(\sigma_{1},\sigma_{2},\sigma_{3}\right)\right)\right)$
is less than $2r$. Then each $\conv\left(V\left(r,\left(\sigma_{1},\sigma_{2},\sigma_{3}\right)\right)\right)$
can contain at most one point of $P_{3}\left(r\right)$, and since
$S_{3}^{*}\backslash S_{3}\left(r\right)=\bigcup_{\sigma_{1},\sigma_{2},\sigma_{3}\in\left\{ -1,1\right\} }\conv\left(V\left(r,\left(\sigma_{1},\sigma_{2},\sigma_{3}\right)\right)\right)$,
the number of points of $P_{3}\left(r\right)$ in $S_{3}^{*}\backslash S_{3}\left(r\right)$
is bounded above by $8$.
\begin{lem}
Let $r\in\left(\frac{1}{2},\frac{4}{7}\right]$ and $\sigma_{1},\sigma_{2},\sigma_{3}\in\left\{ -1,1\right\} $.
For any two points $\mathbf{x},\mathbf{y}\in\conv\left(V\left(r,\left(\sigma_{1},\sigma_{2},\sigma_{3}\right)\right)\right)$,
$\left|\left|\mathbf{x}-\mathbf{y}\right|\right|_{1}<2r$.
\end{lem}

\begin{proof}
Without loss of generality, let $\sigma_{1}=\sigma_{2}=\sigma_{3}=1$,
then $\mathbf{x},\mathbf{y}\in\conv\left(V\left(r,\left(1,1,1\right)\right)\right)$.
It suffices to show that the distance between any two points in $V\left(r,\left(1,1,1\right)\right)$
is less than $2r$, then the conclusion for all points in $\conv\left(V\left(r,\left(1,1,1\right)\right)\right)$
follows by the convexity of $\conv\left(V\left(r,\left(1,1,1\right)\right)\right)$.
We also assume that the two points are distinct. Suppose that neither
point is $\left(2r-1,2r-1,2r-1\right)^{\mathsf{T}}$, where $\mathbf{v}^{\mathsf{T}}$
is the transpose of $\mathbf{v}$, then both points are permutations
of $\left(1-r,1-r,2r-1\right)^{\mathsf{T}}$, so the distance between
the two points is 
\begin{eqnarray*}
\left|\left|\mathbf{x}-\mathbf{y}\right|\right|_{1} & = & 0+\left|\left(1-r\right)-\left(2r-1\right)\right|+\left|\left(2r-1\right)-\left(1-r\right)\right|\\
 & = & 0+\left(2-3r\right)+\left(2-3r\right)\\
 & = & 4-6r\\
 & < & 2r.
\end{eqnarray*}
 If one of the points is $\left(2r-1,2r-1,2r-1\right)^{\mathsf{T}}$,
then the other point must be a permutation of $\left(1-r,1-r,2r-1\right)^{\mathsf{T}}$,
so the distance between the two points is 
\begin{eqnarray*}
\left|\left|\mathbf{x}-\mathbf{y}\right|\right|_{1} & = & \left|\left(1-r\right)-\left(2r-1\right)\right|+\left|\left(2r-1\right)-\left(1-r\right)\right|\\
 & < & 2r.
\end{eqnarray*}
\end{proof}
Below is the proof for Theorem 1.2 (c).
\begin{proof}[\emph{Proof of Theorem 1.2 (c)}]
 Let $r\in\left(\frac{1}{2},\frac{4}{7}\right]$. By Lemma 3.3, 
\[
\left|P_{3}\left(r\right)\cap S_{3}\left(r\right)\right|\leq6.
\]
 Write $P_{3}\left(r\right)$ as the union of two sets $P_{3}\left(r\right)\cap S_{3}\left(r\right)$
and $P_{3}\left(r\right)\cap\left(S_{3}^{*}\backslash S_{3}\left(r\right)\right)$,
whose cardinalities can be individually bounded above. In particular,
by Lemma 4.1 the latter set can be expressed as 
\begin{eqnarray*}
\left|P_{3}\left(r\right)\right| & \leq & \left|P_{3}\left(r\right)\cap S_{3}\left(r\right)\right|+\left|P_{3}\left(r\right)\cap\left(S_{3}^{*}\backslash S_{3}\left(r\right)\right)\right|\\
 & = & 6+\left|P_{3}\left(r\right)\cap\left(\bigcup_{\sigma_{1},\sigma_{2},\sigma_{3}\in\left\{ -1,1\right\} }\conv\left(V\left(r,\left(\sigma_{1},\sigma_{2},\sigma_{3}\right)\right)\right)\right)\right|\\
 & \leq & 6+\sum_{\sigma_{1},\sigma_{2},\sigma_{3}\in\left\{ -1,1\right\} }\left|P_{3}\left(r\right)\cap\conv\left(V\left(r,\left(\sigma_{1},\sigma_{2},\sigma_{3}\right)\right)\right)\right|.
\end{eqnarray*}
 An immediate consequence of Lemma 4.5 is that 
\[
\left|P_{3}\left(r\right)\cap\conv\left(V\left(r,\left(\sigma_{1},\sigma_{2},\sigma_{3}\right)\right)\right)\right|\leq1
\]
 for all $\sigma_{1},\sigma_{2},\sigma_{3}\in\left\{ -1,1\right\} $,
which, when combined with the previous inequality, gives 
\begin{eqnarray*}
\left|P_{3}\left(r\right)\right| & \leq & 6+8\\
 & = & 14.
\end{eqnarray*}
 This inequality holds for any $P_{3}\left(r\right)$, so 
\[
\gamma\left(C_{3}^{*},r\right)\leq14\qquad\text{for }r\in\left(\frac{1}{2},\frac{4}{7}\right]\text{.}
\]
\end{proof}
We are not able to find the exact value of $\gamma\left(C_{3}^{*},r\right)$
for such $r$, but some lower bounds are in Section 5.

\section{Constructive lower bounds including the proof of Proposition 1.3}

In contrast to the upper bounds, the lower bounds are all obtained
by explicit constructions of points in the cross-polytope. For $n=3$
and $r\in\left(\frac{1}{2},\frac{2}{3}\right]$, all of the constructions
shown here contain the six points of $V_{3}$ and the remaining points
are in the union of the eight sets $\conv\left(V\left(r,\left(\sigma_{1},\sigma_{2},\sigma_{3}\right)\right)\right)$.
There are no claims of uniqueness made here; more than one set of
points may achieve the lower bounds of Theorem 1.3.

The calculations in the proofs below can be performed by hand or using
a computer.
\begin{prop}
Let $\mathbf{q}_{n}=\left(\frac{1}{n},\ldots,\frac{1}{n}\right)^{\mathsf{T}}\in\mathbb{R}^{n}$.
Then $V_{n}\cup\left\{ \pm\mathbf{q}_{n}\right\} \subset C_{3}^{*}$
and for $r\in\left(0,1-\frac{1}{n}\right]$, 
\[
V_{n}\cup\left\{ \pm\mathbf{q}_{n}\right\} =\left\{ \begin{pmatrix}1\\
0\\
0\\
\vdots\\
0
\end{pmatrix},\begin{pmatrix}-1\\
0\\
0\\
\vdots\\
0
\end{pmatrix},\begin{pmatrix}0\\
1\\
0\\
\vdots\\
0
\end{pmatrix},\begin{pmatrix}0\\
-1\\
0\\
\vdots\\
0
\end{pmatrix},\ldots,\begin{pmatrix}0\\
0\\
\vdots\\
0\\
1
\end{pmatrix},\begin{pmatrix}0\\
0\\
\vdots\\
0\\
-1
\end{pmatrix}\right\} \cup\frac{1}{n}\left\{ \begin{pmatrix}1\\
1\\
\vdots\\
1\\
1
\end{pmatrix},-\begin{pmatrix}1\\
1\\
\vdots\\
1\\
1
\end{pmatrix}\right\} 
\]
 is a packing set of $rC_{n}^{*}$.
\end{prop}

\begin{proof}
Any points $\mathbf{x},\mathbf{y}\in V_{n}\cup\left\{ \pm\mathbf{q}_{n}\right\} $,
$\mathbf{x}\neq\mathbf{y}$, have the property that $\left|\left|\mathbf{x}\right|\right|_{1}\leq1$
and $\left|\left|\mathbf{x}-\mathbf{y}\right|\right|_{1}\leq2\left(1-\frac{1}{n}\right)$,
so $V_{n}\cup\left\{ \pm\mathbf{q}_{n}\right\} \subset C_{3}^{*}$
is a packing set of $rC_{n}^{*}$ for $r\leq1-\frac{1}{n}$.
\end{proof}
\begin{prop}
Let 
\[
Q_{10}=\frac{1}{3}\left\{ \begin{pmatrix}1\\
1\\
1
\end{pmatrix},\begin{pmatrix}-1\\
-1\\
1
\end{pmatrix},\begin{pmatrix}-1\\
1\\
-1
\end{pmatrix},\begin{pmatrix}1\\
-1\\
-1
\end{pmatrix}\right\} .
\]
 Then $V_{3}\cup Q_{10}\subset C_{3}^{*}$ and for $r\in\left(0,\frac{2}{3}\right]$,
\[
V_{3}\cup Q_{10}=\left\{ \begin{pmatrix}1\\
0\\
0
\end{pmatrix},\begin{pmatrix}-1\\
0\\
0
\end{pmatrix},\begin{pmatrix}0\\
1\\
0
\end{pmatrix},\begin{pmatrix}0\\
-1\\
0
\end{pmatrix},\begin{pmatrix}0\\
0\\
1
\end{pmatrix},\begin{pmatrix}0\\
0\\
-1
\end{pmatrix}\right\} \cup\frac{1}{3}\left\{ \begin{pmatrix}1\\
1\\
1
\end{pmatrix},\begin{pmatrix}-1\\
-1\\
1
\end{pmatrix},\begin{pmatrix}-1\\
1\\
-1
\end{pmatrix},\begin{pmatrix}1\\
-1\\
-1
\end{pmatrix}\right\} 
\]
 is a packing set of $rC_{3}^{*}$.
\end{prop}

\begin{proof}
Any points $\mathbf{x},\mathbf{y}\in V_{3}\cup Q_{10}$, $\mathbf{x}\neq\mathbf{y}$,
have the property that $\left|\left|\mathbf{x}\right|\right|_{1}\leq1$
and $\left|\left|\mathbf{x}-\mathbf{y}\right|\right|_{1}\leq\frac{4}{3}$,
so $V_{3}\cup Q_{10}\subset C_{3}^{*}$ is a packing set of $rC_{n}^{*}$
for $r\leq\frac{2}{3}$.
\end{proof}
\begin{prop}
Let 
\[
Q_{12}^{+}=\frac{1}{5}\left\{ \begin{pmatrix}2\\
2\\
1
\end{pmatrix},\begin{pmatrix}-2\\
1\\
2
\end{pmatrix},\begin{pmatrix}1\\
-2\\
2
\end{pmatrix}\right\} .
\]
 Then $V_{3}\cup Q_{12}^{+}\cup\left(-Q_{12}^{+}\right)\subset C_{3}^{*}$
and for $r\in\left(0,\frac{3}{5}\right]$, 
\begin{eqnarray*}
V_{3}\cup Q_{12}^{+}\cup\left(-Q_{12}^{+}\right) & = & \left\{ \begin{pmatrix}1\\
0\\
0
\end{pmatrix},\begin{pmatrix}-1\\
0\\
0
\end{pmatrix},\begin{pmatrix}0\\
1\\
0
\end{pmatrix},\begin{pmatrix}0\\
-1\\
0
\end{pmatrix},\begin{pmatrix}0\\
0\\
1
\end{pmatrix},\begin{pmatrix}0\\
0\\
-1
\end{pmatrix}\right\} \\
 &  & \,\cup\,\frac{1}{5}\left\{ \begin{pmatrix}2\\
2\\
1
\end{pmatrix},\begin{pmatrix}-2\\
1\\
2
\end{pmatrix},\begin{pmatrix}1\\
-2\\
2
\end{pmatrix}\right\} \cup-\frac{1}{5}\left\{ \begin{pmatrix}2\\
2\\
1
\end{pmatrix},\begin{pmatrix}-2\\
1\\
2
\end{pmatrix},\begin{pmatrix}1\\
-2\\
2
\end{pmatrix}\right\} 
\end{eqnarray*}
 is a packing set of $rC_{3}^{*}$.
\end{prop}

\begin{proof}
Any points $\mathbf{x},\mathbf{y}\in V_{3}\cup Q_{12}^{+}\cup\left(-Q_{12}^{+}\right)$,
$\mathbf{x}\neq\mathbf{y}$, have the property that $\left|\left|\mathbf{x}\right|\right|_{1}\leq1$
and $\left|\left|\mathbf{x}-\mathbf{y}\right|\right|_{1}\leq\frac{6}{5}$,
so $V_{3}\cup Q_{12}^{+}\cup\left(-Q_{12}^{+}\right)\subset C_{3}^{*}$
is a packing set of $rC_{n}^{*}$ for $r\leq\frac{3}{5}$.
\end{proof}
Finally we consider the case $r\in\left(0,\frac{6}{11}\right]$. The
construction below differs from the previous constructions as there
are no obvious large-scale symmetries.
\begin{prop}
Let 
\[
Q_{13}=\frac{1}{11}\left\{ \begin{pmatrix}-1\\
5\\
5
\end{pmatrix},\begin{pmatrix}5\\
-1\\
5
\end{pmatrix},\begin{pmatrix}5\\
5\\
-1
\end{pmatrix},\begin{pmatrix}-5\\
-2\\
4
\end{pmatrix},\begin{pmatrix}-5\\
4\\
-2
\end{pmatrix},\begin{pmatrix}4\\
-2\\
-5
\end{pmatrix},\begin{pmatrix}-3\\
-5\\
-3
\end{pmatrix}\right\} .
\]
 Then $V_{3}\cup Q_{13}\subset C_{3}^{*}$ and for $r\leq\frac{6}{11}$,
\begin{eqnarray*}
V_{3}\cup Q_{13} & = & \left\{ \begin{pmatrix}1\\
0\\
0
\end{pmatrix},\begin{pmatrix}-1\\
0\\
0
\end{pmatrix},\begin{pmatrix}0\\
1\\
0
\end{pmatrix},\begin{pmatrix}0\\
-1\\
0
\end{pmatrix},\begin{pmatrix}0\\
0\\
1
\end{pmatrix},\begin{pmatrix}0\\
0\\
-1
\end{pmatrix}\right\} \\
 &  & \,\cup\,\frac{1}{11}\left\{ \begin{pmatrix}-1\\
5\\
5
\end{pmatrix},\begin{pmatrix}5\\
-1\\
5
\end{pmatrix},\begin{pmatrix}5\\
5\\
-1
\end{pmatrix},\begin{pmatrix}-5\\
-2\\
4
\end{pmatrix},\begin{pmatrix}-5\\
4\\
-2
\end{pmatrix},\begin{pmatrix}4\\
-2\\
-5
\end{pmatrix},\begin{pmatrix}-3\\
-5\\
-3
\end{pmatrix}\right\} 
\end{eqnarray*}
 is a packing set of $rC_{3}^{*}$.
\end{prop}

\begin{proof}
Any points $\mathbf{x},\mathbf{y}\in V_{3}\cup Q_{13}$, $\mathbf{x}\neq\mathbf{y}$,
have the property that $\left|\left|\mathbf{x}\right|\right|_{1}\leq1$
and $\left|\left|\mathbf{x}-\mathbf{y}\right|\right|_{1}\leq\frac{12}{11}$,
so $V_{3}\cup Q_{13}\subset C_{3}^{*}$ is a packing set of $rC_{n}^{*}$
for $r\leq\frac{6}{11}$.
\end{proof}
We do not know if this result can be improved, either in the sense
of a $13$-point configuration for some $r>\frac{6}{11}$ or a $14$-point
configuration for $r=\frac{6}{11}$. Regarding the first avenue for
improvement, the upper end of the range $r\in\left(0,\frac{6}{11}\right]$
cannot be raised without moving the points of $V_{3}\cup Q_{13}$.
As for the second, according to the proof of Theorem $1.2$ (c), at
most eight points in any $P_{3}\left(\frac{6}{11}\right)$ can be
in the sets $\conv\left(V\left(\frac{6}{11},\left(\sigma_{1},\sigma_{2},\sigma_{3}\right)\right)\right)$
for all $\sigma_{1},\sigma_{2},\sigma_{3}\in\left\{ -1,1\right\} $.
The packing set $V_{3}\cup Q_{13}$ contains points in each set of
the form $\conv\left(V\left(\frac{6}{11},\left(\sigma_{1},\sigma_{2},\sigma_{3}\right)\right)\right)$,
$\sigma_{1},\sigma_{2},\sigma_{3}\in\left\{ -1,1\right\} $, except
for $\conv\left(V\left(\frac{6}{11},\left(1,1,1\right)\right)\right)$,
see Figure 6.5. Since 
\[
V\left(\frac{6}{11},\left(1,1,1\right)\right)=\left\{ \begin{pmatrix}0\\
0\\
0
\end{pmatrix},\frac{1}{11}\begin{pmatrix}5\\
5\\
1
\end{pmatrix},\frac{1}{11}\begin{pmatrix}5\\
1\\
5
\end{pmatrix},\frac{1}{11}\begin{pmatrix}1\\
5\\
5
\end{pmatrix}\right\} 
\]
 and the distances from each point in this set to $\frac{1}{11}\left(-1,5,5\right)^{\mathsf{T}}\in V_{3}\cup Q_{13}$
are 
\begin{eqnarray*}
\left|\left|\begin{pmatrix}0\\
0\\
0
\end{pmatrix}-\frac{1}{11}\begin{pmatrix}-1\\
5\\
5
\end{pmatrix}\right|\right|_{1} & = & 1\\
 & < & \frac{12}{11},
\end{eqnarray*}
 
\begin{eqnarray*}
\left|\left|\frac{1}{11}\begin{pmatrix}5\\
5\\
1
\end{pmatrix}-\frac{1}{11}\begin{pmatrix}-1\\
5\\
5
\end{pmatrix}\right|\right|_{1} & = & \left|\frac{5}{11}+\frac{1}{11}\right|+\left|\frac{1}{11}-\frac{5}{11}\right|+\left|\frac{5}{11}-\frac{5}{11}\right|\\
 & = & \frac{10}{11}\\
 & < & \frac{12}{11},
\end{eqnarray*}
 
\[
\left|\left|\frac{1}{11}\begin{pmatrix}5\\
1\\
5
\end{pmatrix}-\frac{1}{11}\begin{pmatrix}-1\\
5\\
5
\end{pmatrix}\right|\right|_{1}<\frac{12}{11},
\]
 and 
\begin{eqnarray*}
\left|\left|\frac{1}{11}\begin{pmatrix}1\\
5\\
5
\end{pmatrix}-\frac{1}{11}\begin{pmatrix}-1\\
5\\
5
\end{pmatrix}\right|\right|_{1} & = & \left|\frac{1}{11}+\frac{1}{11}\right|+\left|\frac{5}{11}-\frac{5}{11}\right|+\left|\frac{5}{11}-\frac{5}{11}\right|\\
 & = & \frac{2}{11}\\
 & < & \frac{12}{11},
\end{eqnarray*}
it follows that the distance from any point in $\conv\left(V\left(\frac{6}{11},\left(1,1,1\right)\right)\right)$
to $\conv\left(V\left(\frac{6}{11},\left(1,1,1\right)\right)\right)$
is less than $1$. Therefore, a $14$-point packing set of $\frac{6}{11}C_{3}^{*}$
is not possible without moving one or more of the points in the subset
\[
\left\{ \begin{pmatrix}1\\
0\\
0
\end{pmatrix},\begin{pmatrix}0\\
1\\
0
\end{pmatrix},\begin{pmatrix}0\\
0\\
1
\end{pmatrix}\right\} \cup\frac{1}{11}\left\{ \begin{pmatrix}-1\\
5\\
5
\end{pmatrix},\begin{pmatrix}5\\
-1\\
5
\end{pmatrix},\begin{pmatrix}5\\
5\\
-1
\end{pmatrix}\right\} \subset V_{3}\cup Q_{13}.
\]

\begin{proof}[\emph{Proof of Proposition 1.3}]
\emph{} By Proposition 5.4, the set $V_{3}\cup Q_{13}$ is a subset
of $C_{n}^{*}$ with $13$ points and is a packing set for $rC_{3}^{*}$
where $r\in\left(\frac{1}{2},\frac{6}{11}\right]$. Therefore 
\[
\gamma\left(C_{3}^{*},r\right)\geq13\qquad\text{for }r\in\left(\frac{1}{2},\frac{6}{11}\right]\text{.}
\]
\end{proof}

\section{Diagrams of cross-polytope packings}

Below are graphs showing the unit cross-polytope with cross-polytopes
of radius $r$ around each point of $V_{3}$, $V_{3}\cup Q_{10}$,
$V_{3}\cup Q_{12}^{+}\cup\left(-Q_{12}^{+}\right)$, and $V_{3}\cup Q_{13}$.
For each diagram except the first one, the value of $r$ in the diagram
is the largest possible for that configuration of points. In each
diagram the grey cross-polytope in the middle is $C_{3}^{*}$.

\paragraph{$\boldsymbol{V_{3}}$: $\boldsymbol{6}$ points in $\boldsymbol{C_{3}^{*}}$}

$V_{3}$ is a packing set of $rC_{3}^{*}$ for all $0<r\leq1$.
\noindent \begin{center}
\includegraphics[scale=0.25]{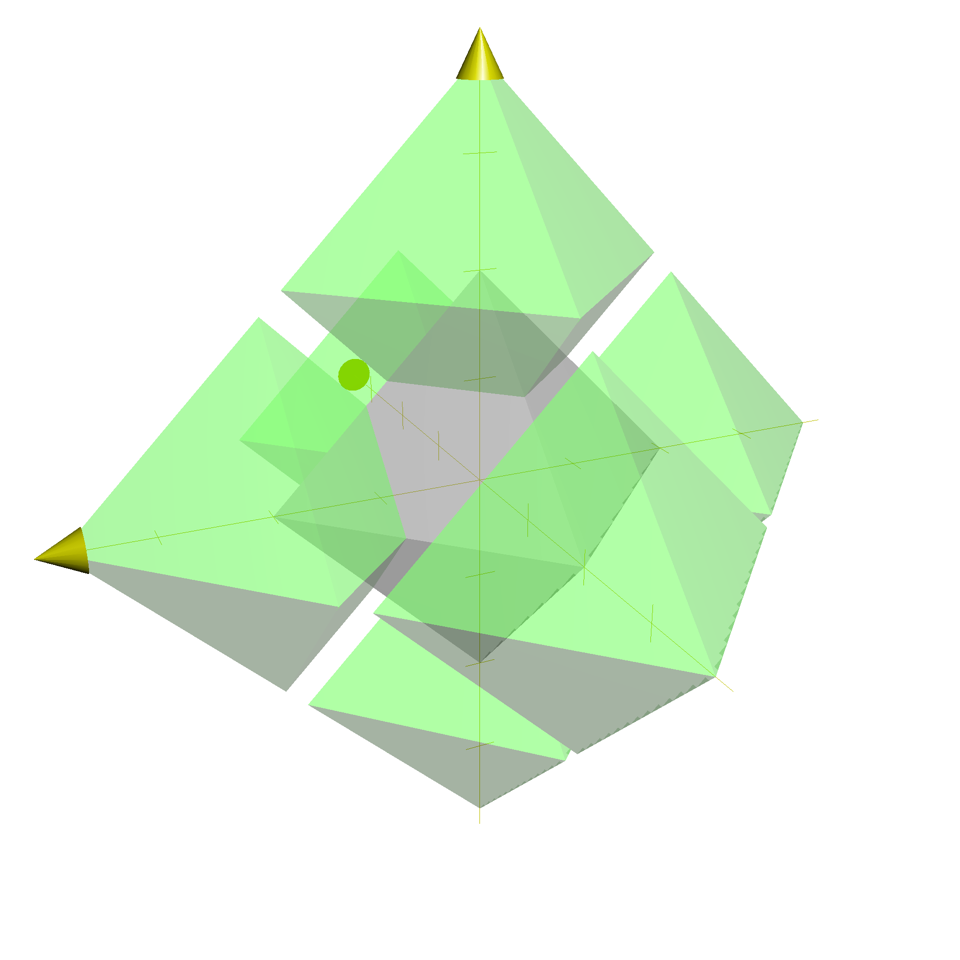}
\par\end{center}

\noindent \begin{center}
\emph{Figure 6.1}. The green cross-polytopes represent the sets $\overline{C\left(\mathbf{x},\frac{9}{10}\right)}$
for $\mathbf{x}\in V_{3}$.
\par\end{center}

\paragraph{$\boldsymbol{V_{3}\cup Q_{10}}$: $\boldsymbol{10}$ points in $\boldsymbol{C_{3}^{*}}$}

$V_{3}\cup Q_{10}$ is a packing set of $rC_{3}^{*}$ for all $0<r\leq\frac{2}{3}$.
\noindent \begin{center}
\includegraphics[scale=0.25]{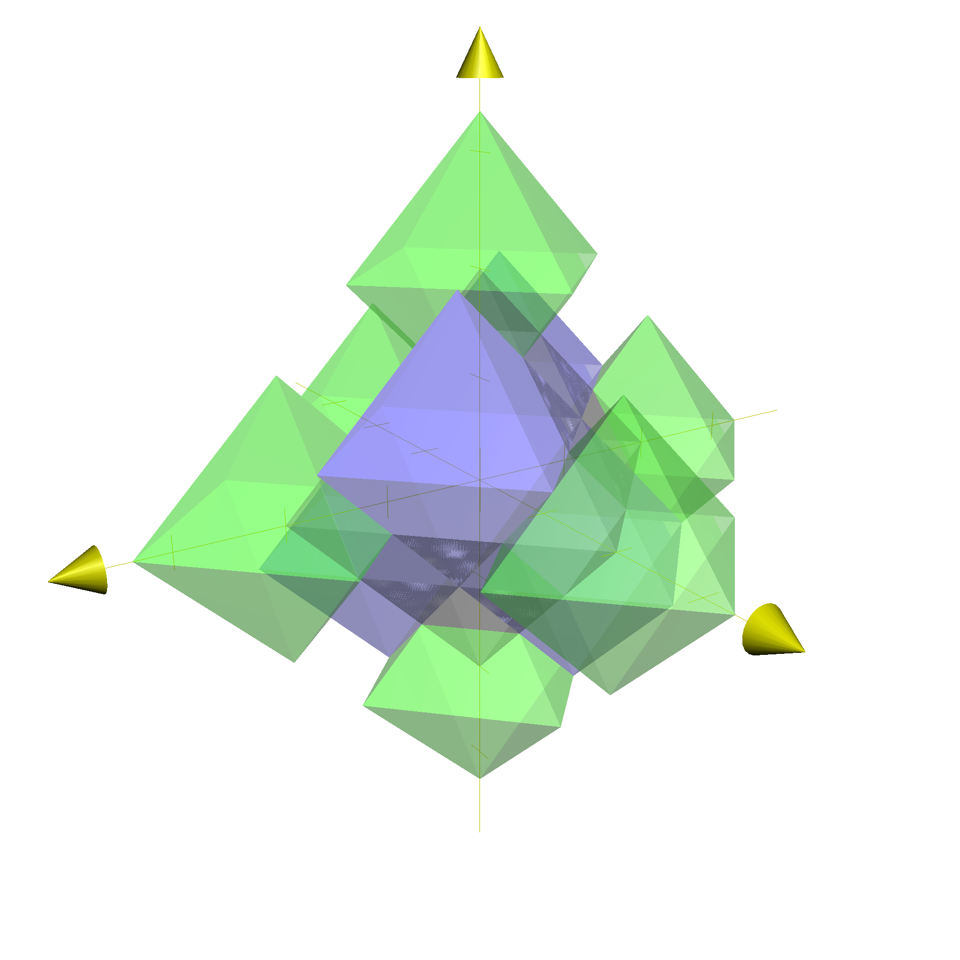}\includegraphics[scale=0.25]{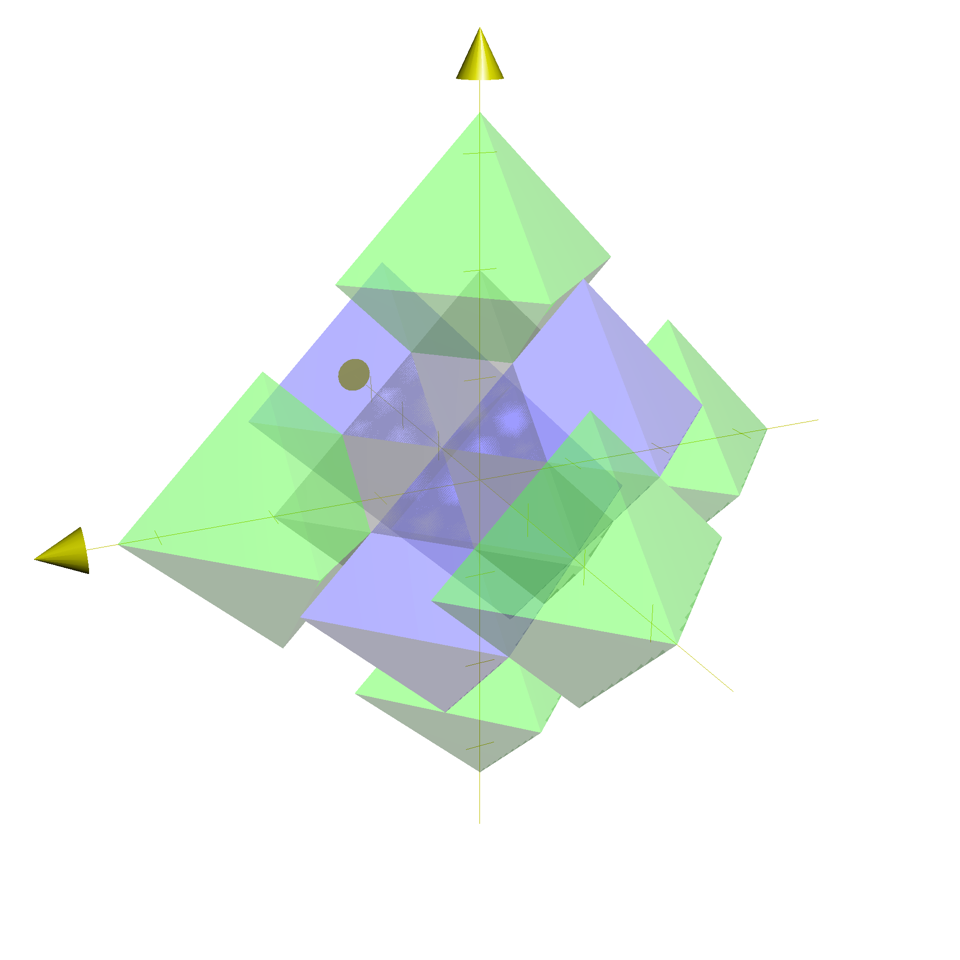}
\par\end{center}

\noindent \begin{center}
\emph{Figures 6.2 (left) and 6.3 (right)}. The green cross-polytopes
represent the sets $\overline{C\left(\mathbf{x},\frac{2}{3}\right)}$
for $\mathbf{x}\in V_{3}$ and the blue cross-polytopes represent
the sets $\overline{C\left(\mathbf{y},\frac{2}{3}\right)}$ for $\mathbf{y}\in Q_{10}$.
\par\end{center}

\paragraph{$\boldsymbol{V_{3}\cup Q_{12}^{+}\cup\left(-Q_{12}^{+}\right)}$:
$\boldsymbol{12}$ points in $\boldsymbol{C_{3}^{*}}$}

$V_{3}\cup Q_{12}^{+}\cup\left(-Q_{12}^{+}\right)$ is a packing set
of $rC_{3}^{*}$ for all $0<r\leq\frac{3}{5}$.
\noindent \begin{center}
\includegraphics[scale=0.25]{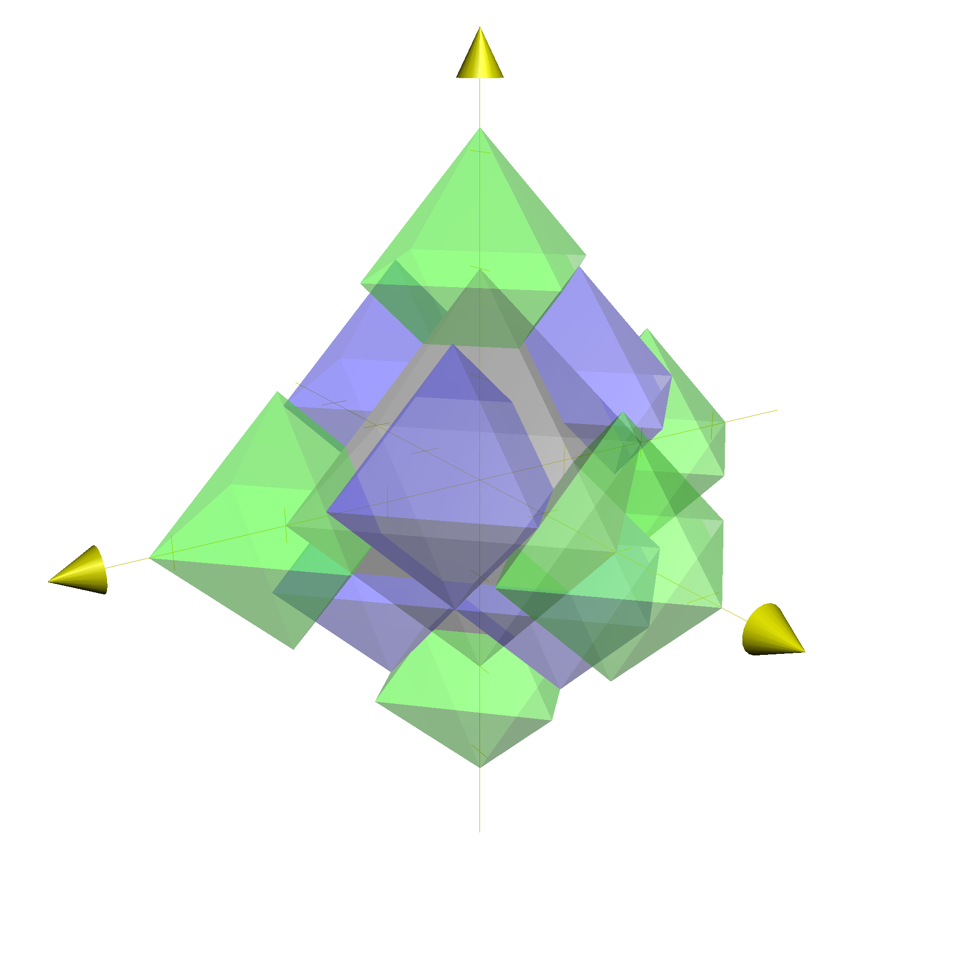}\includegraphics[scale=0.25]{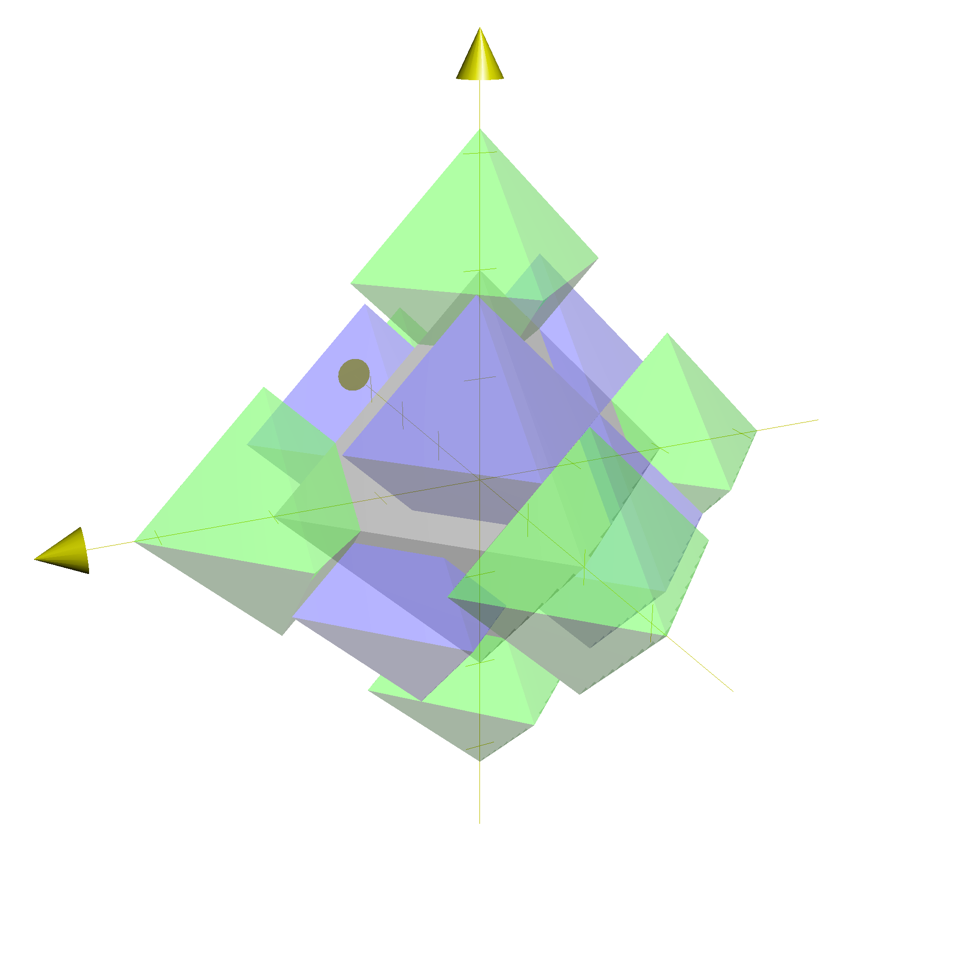}
\par\end{center}

\noindent \begin{center}
\emph{Figures 6.4 (left) and 6.5 (right)}. The green cross-polytopes
represent the sets $\overline{C\left(\mathbf{x},\frac{3}{5}\right)}$
for $\mathbf{x}\in V_{3}$ and the blue cross-polytopes represent
the sets $\overline{C\left(\mathbf{y},\frac{3}{5}\right)}$ for $\mathbf{y}\in V_{3}\cup Q_{12}^{+}\cup\left(-Q_{12}^{+}\right)$.
\par\end{center}

\paragraph{$\boldsymbol{V_{3}\cup Q_{13}}$: $\boldsymbol{13}$ points in $\boldsymbol{C_{3}^{*}}$}

$V_{3}\cup Q_{13}$ is a packing set of $rC_{3}^{*}$ for all $0<r\leq\frac{6}{11}$.
\noindent \begin{center}
\includegraphics[scale=0.25]{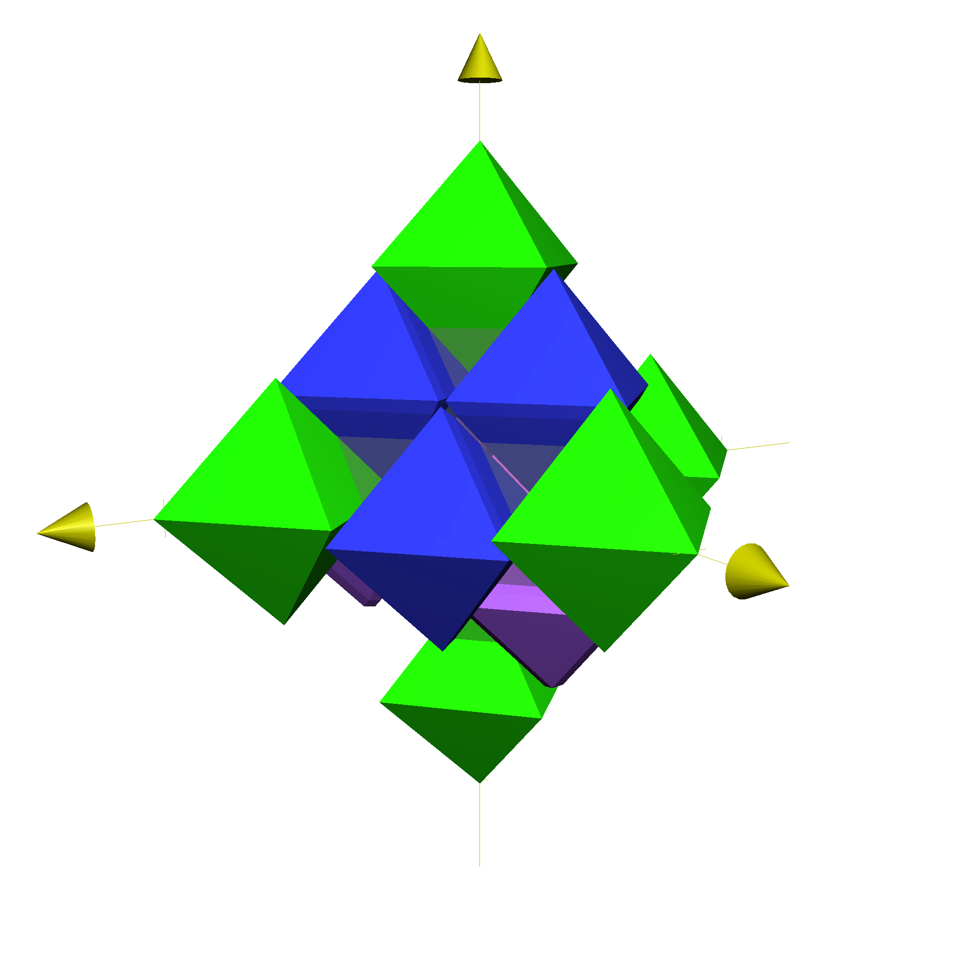}\includegraphics[scale=0.25]{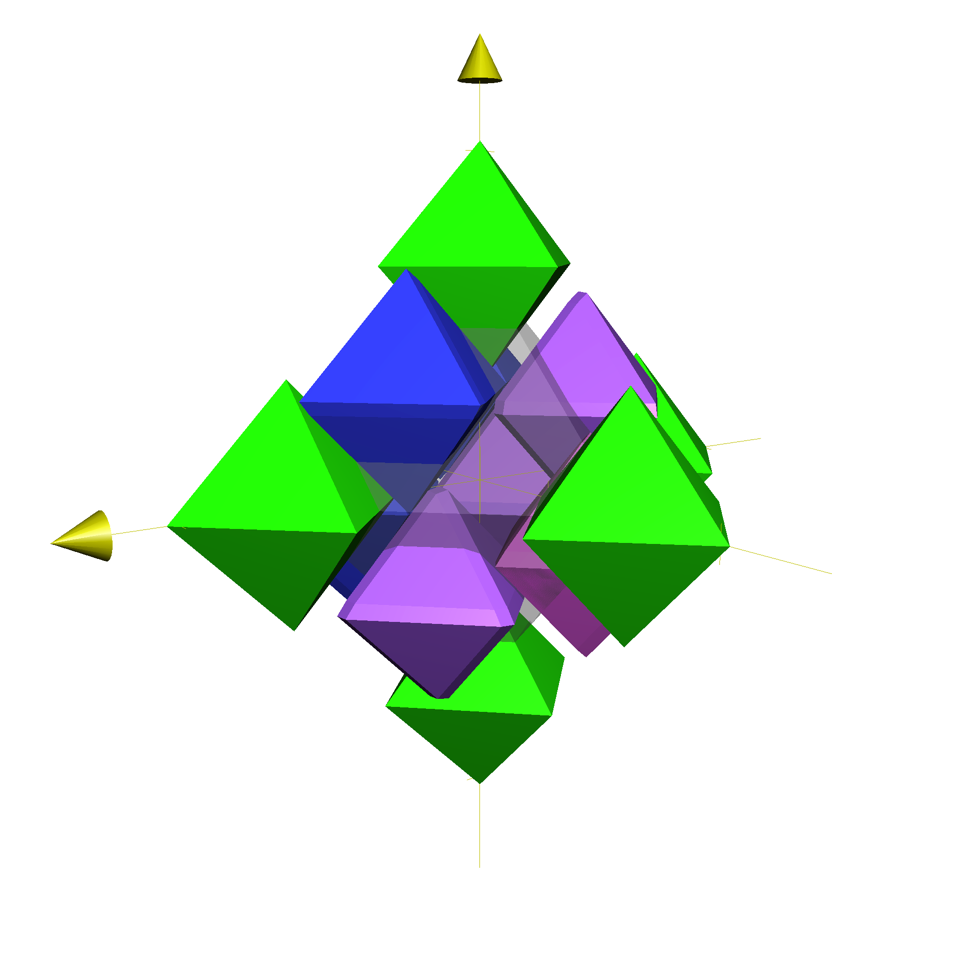}\\
\includegraphics[scale=0.25]{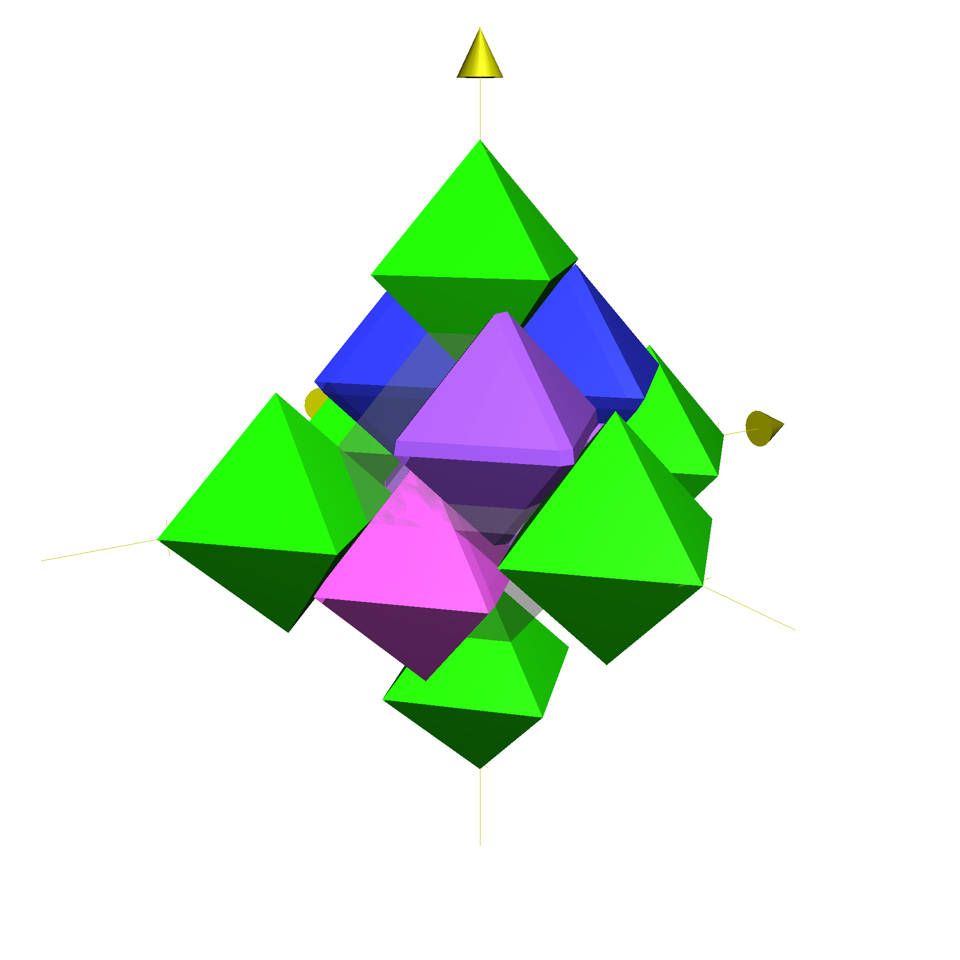}\includegraphics[scale=0.25]{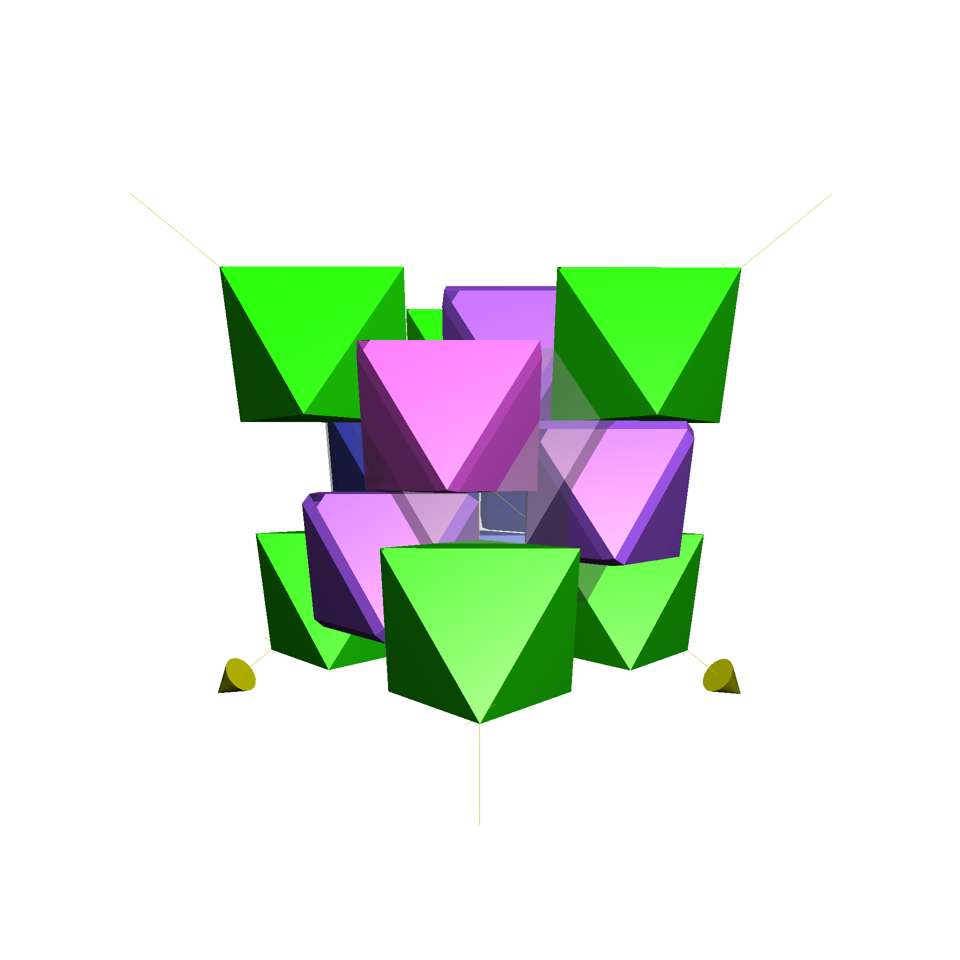}
\par\end{center}

\noindent \begin{center}
\emph{Figures 6.6 (top left), 6.7 (top right), 6.8 (bottom left),
and 6.9 (bottom right)}. The green cross-polytopes represent the sets
$\overline{C\left(\mathbf{x},\frac{6}{11}\right)}$ for $\mathbf{x}\in V_{3}$,
the blue cross-polytopes represent $\overline{C\left(\frac{1}{11}\left(-1,5,5\right),\frac{6}{11}\right)}$,
$\overline{C\left(\frac{1}{11}\left(5,-1,5\right),\frac{6}{11}\right)}$,
and $\overline{C\left(\frac{1}{11}\left(5,5,-1\right),\frac{6}{11}\right)}$,
the purple cross-polytopes represent $\overline{C\left(\frac{1}{11}\left(-5,-2,4\right),\frac{6}{11}\right)}$,
$\overline{C\left(\frac{1}{11}\left(-5,4,-2\right),\frac{6}{11}\right)}$,
and $\overline{C\left(\frac{1}{11}\left(4,-2,-5\right),\frac{6}{11}\right)}$,
and the magenta cross-polytope represents $\overline{C\left(\frac{1}{11}\left(-3,-5,-3\right),\frac{6}{11}\right)}$.
\par\end{center}

\section{Acknowledgements}

Thanks to Martin Henk for his advice, suggestions, and feedback, and
Fei Xue for feedback and discussions, especially with organizing and
simplifying the proof of the $n=3$ and $r\in\left(\frac{3}{5},\frac{2}{3}\right]$
case.

\bibliographystyle{siam}
\bibliography{crosspolytopepackings}

\lyxaddress{Institut f{\"u}r Mathematik, Technische Universit{\"a}t Berlin, Sekr. MA
4-1, Stra{\ss}e des 17. Juni 136, 10623 Berlin, Germany}

\noindent E-mail: $\mathtt{chun@math.tu\text{-}berlin.de}$
\end{document}